\documentclass[12pt]{article}
\usepackage[margin=2cm]{geometry}
\usepackage[dvipsnames,svgnames]{xcolor}
\usepackage{latexsym,hyperref,amssymb,amsfonts,amsmath,amsthm}

\numberwithin{equation}{section}
\newtheorem{theorem}{Theorem}[section]

\newtheorem{lemma}[theorem]{Lemma}

\def \Z {\mathbb{Z}}

\def \mod {\ ({\rm mod}\,\,}
\def \deg {{\rm deg}}
\def \lcm {{\rm lcm}}

\newcommand{\vect}[1]{\boldsymbol{#1}}

\title{\bf On balanced incomplete block designs with specified weak chromatic number}

\author{
Daniel Horsley\\
School of Mathematical Sciences \\
Monash University \\
Vic 3800, Australia \\[0.1cm]
\texttt{danhorsley@gmail.com}\\[0.4cm]
David A. Pike\\
Department of Mathematics and Statistics \\
Memorial University of Newfoundland \\
St. John's, NL, Canada A1C 5S7 \\[0.1cm]
\texttt{dapike@mun.ca}}

\date{ }

\colorlet{g}{DarkGreen}
\colorlet{r}{DarkRed}

\begin{document}
\maketitle \thispagestyle{empty}
\def\baselinestretch{1.5}\small\normalsize

\begin{abstract}
A weak $c$-colouring of a balanced incomplete block design (BIBD) is a colouring of the points of
the design with $c$ colours in such a way that no block of the design has all of its vertices
receive the same colour. A BIBD is said to be weakly $c$-chromatic if $c$ is the smallest number of
colours with which the design can be weakly coloured. In this paper we show that for all $c \geq 2$
and $k \geq 3$ with $(c,k) \neq (2,3)$, the obvious necessary conditions for the existence of a
$(v,k,\lambda)$-BIBD are asymptotically sufficient for the existence of a weakly $c$-chromatic
$(v,k,\lambda)$-BIBD.
\end{abstract}

\section{Introduction}

A \emph{balanced incomplete block design} of order $v$, block size $k$ and index $\lambda$, denoted
a $(v,k,\lambda)$-BIBD, is a pair $(V,\mathcal{B})$ such that $V$ is a set of $v$ elements (called
points) and $\mathcal{B}$ is a collection of $k$ element subsets of $V$ (called blocks) such that
each unordered pair of points in $V$ is contained in exactly $\lambda$ blocks in $\mathcal{B}$. A
\emph{partial} $(v,k,\lambda)$-BIBD is defined similarly except that each pair of points in $V$
must be contained in at most $\lambda$ blocks in $\mathcal{B}$.

For a positive integer $c$, a weak $c$-colouring of a (partial) $(v,k,\lambda)$-BIBD is a colouring
of the points of the design with $c$ colours in such a way that no block of the design has all of
its points receive the same colour. A (partial) $(v,k,\lambda)$-BIBD is said to be \emph{weakly
$c$-chromatic}, or to have \emph{weak chromatic number} $c$, if $c$ is the smallest number of
colours with which the design can be weakly coloured. Since weak colourings are the only colourings
of designs we will consider in this paper, we will often omit the adjectives `weak' and `weakly' in
what follows.

It is obvious that if there exists a $(v,k,\lambda)$-BIBD then
\begin{itemize}
    \item[(i)]
$\lambda(v-1) \equiv 0 \mod k-1)$; and
    \item[(ii)]
$\lambda v(v-1) \equiv 0 \mod k(k-1))$.
\end{itemize}
Wilson \cite{Wi} famously proved that (i) and (ii) are asymptotically sufficient for the existence
of a $(v,k,\lambda)$-BIBD. That is, for any positive integers $k$ and $\lambda$ with $k \geq 3$,
there exists an integer $N'(k,\lambda)$ such that if $v \geq N'(k,\lambda)$ then (i) and (ii) are
sufficient for the existence of a $(v,k,\lambda)$-BIBD.  In this paper, we will extend Wilson's
result to $c$-chromatic BIBDs by showing that, for any positive integers $c$, $k$ and $\lambda$,
such that $c \geq 2$, $k \geq 3$ and $(k,c) \neq (3,2)$, (i) and (ii) are asymptotically sufficient
for the existence of a $c$-chromatic $(v,k,\lambda)$-BIBD. For the sake of brevity, we will call
positive integers $v$ which satisfy (i) and (ii) \emph{$(k,\lambda)$-admissible}. Note that if an
integer $v$ is $(k,\lambda)$-admissible then so is every positive integer congruent to $v$ modulo
$k(k-1)$.

Weak colourings were first introduced in the context of hypergraphs, and this naturally led to the
study of weak colourings of block designs. A simple counting argument \cite{Ro} shows that
$2$-chromatic $(v,3,\lambda)$-BIBDs exist only for $v \leq 4$. For a positive integer $\lambda$ it
is known that a $2$-chromatic $(v,4,\lambda)$-BIBD exists for each $(4,\lambda)$-admissible
integer $v$, with almost all of the problem solved in \cite{HoLiPh1} and \cite{HoLiPh2} and the
outstanding cases resolved in \cite{RoCo} and \cite{FrGrLiRo}. Ling \cite{Li} has proved that a
$2$-chromatic $(v,5,1)$-BIBD exists for each $(5,1)$-admissible integer $v$. It has been shown by
de Brandes, Phelps and R\"{o}dl \cite{DePhRo} that for all integers $c \geq 3$ there is an integer
$N(c,3,1)$ such that for all $(3,1)$-admissible integers $v \geq N(c,3,1)$ there is a
$c$-chromatic $(v,3,1)$-BIBD. The analogous result for $(v,4,1)$-BIBDs has been proved by Linek
and Wantland \cite{LiWa}. For a survey of colourings of block designs see \cite{RoCo}.  The main
result of this paper is as follows.

\begin{theorem}\label{MainTheorem}
Let $c$, $k$ and $\lambda$ be positive integers such that $c \geq 2$, $k \geq 3$ and $(c,k) \neq
(2,3)$. Then there is an integer $N(c,k,\lambda)$ such that there exists a weakly $c$-chromatic
$(v,k,\lambda)$-BIBD for all $(k,\lambda)$-admissible integers $v \geq N(c,k,\lambda)$.
\end{theorem}

In Section 2 we give some definitions that we will require throughout the paper and prove a number
of preliminary results. Sections 3, 4 and 5 deal with BIBDs with block size at least 4. In Section
3 we find various examples of $2$-chromatic BIBDs, and these are then used in Section 4 to obtain
various examples of $c$-chromatic BIBDs for each $c \geq 2$. In Section 5 we are then able to use
results from Sections 2, 3 and 4 to demonstrate the asymptotic existence of $c$-chromatic BIBDs for
each $c \geq 2$. Finally, in Section 6, we deal with the case of BIBDs with block size $3$ (that
is, triple systems) and thus complete the proof of Theorem \ref{MainTheorem}.

\section{Preliminary definitions and results}

Let $v$ and $\lambda$ be positive integers and let $K$ be a set of positive integers. A \emph{group
divisible design} of order $v$ and index $\lambda$ with block sizes from $K$, denoted a
$(K,\lambda)$-GDD, is a triple $(V,\mathcal{G},\mathcal{B})$ such that $V$ is a set of $v$ elements
(called points), $\mathcal{G}$ is a partition of $V$ into parts (called groups) and $\mathcal{B}$
is a collection of subsets of $V$ (called blocks) such that $|B| \in K$ for all $B \in
\mathcal{B}$, each unordered pair of points in different groups is contained in exactly $\lambda$
blocks, and no unordered pair of points in the same group is contained in any block. If, for
integers $g_1,g_2,\ldots,g_t$ and $a_1,a_2,\ldots,a_t$, $\mathcal{G}$ contains $a_i$ groups of size
$g_i$ for $i \in \{1,2,\ldots,t\}$ and $\mathcal{G}$ contains no groups of any other size then we
say that $(V,\mathcal{G},\mathcal{B})$ is of type $g_1^{a_1}g_2^{a_2}\cdots g_t^{a_t}$. We will
abbreviate $(\{k\},\lambda)$-GDD to $(k,\lambda)$-GDD. A $(k,1)$-GDD of type $g^k$ is more commonly
referred to as a transversal design with group size $g$ and block size $k$.

We say that a partial BIBD $(V_1,\mathcal{B}_1)$ is \emph{embedded} in a partial BIBD
$(V_2,\mathcal{B}_2)$ if $V_1 \subseteq V_2$ and $\mathcal{B}_1 \subseteq \mathcal{B}_2$. A
\emph{decomposition} of a graph $G$ is a collection $\{G_1,G_2,\ldots,G_t\}$ of subgraphs of $G$
whose edge sets partition the edge set of $G$. We extend this definition to edge-coloured digraphs
in the obvious way. A $(v,k,\lambda)$-BIBD can be considered as a decomposition of the
$\lambda$-fold complete graph with $v$ vertices into copies of the complete graph with $k$
vertices.

To simplify the presentation of many of our results, we will introduce a generalisation of the
well-known concept of a blocking set. We will say that a collection $\{S_1,S_2,\ldots,S_c\}$ of
pairwise disjoint subsets of the point set of a design is a \emph{blocking system} for that design
if each block of the design has a non-empty intersection with at least two of the sets in
$\{S_1,S_2,\ldots,S_c\}$. We will also refer to such a blocking system as a $c$-blocking system if
we wish to specify the number of sets in the system or as an $(s_1,s_2,\ldots,s_c)$-blocking
system, where $s_i = |S_i|$ for each $i \in \{1,2,\ldots,c\}$, if we wish to specify the sizes of
the sets in the system. Obviously the existence of a $c$-blocking system for a design implies the
existence of a $c$-colouring for that design. Note that if a design on $v$ points has an
$(s_1,s_2,\ldots,s_c)$-blocking system for integers $s_1,s_2,\ldots,s_c$ then it has an
$(s'_1,s'_2,\ldots,s'_c)$-blocking system for all integers $s'_1,s'_2,\ldots,s'_c$ such that $s'_1
+ s'_2 + \cdots + s'_c \leq v$ and $s'_i \geq s_i$ for each $i \in \{1,2,\ldots,c\}$. We will often
use this fact tacitly in what follows.

Finally, we will require a result from \cite{LaWi} on decompositions of edge-coloured graphs. To
state the result we need some additional definitions and notation. For full details of the
framework we refer the reader to \cite{LaWi}. We will denote by $\vect{1_{n}}$ the $n$-dimensional
vector all of whose components are $1$. Let $C$ be a set of colours, let $r=|C|$, and let $\lambda
K^{(C)}_v$ denote the edge-coloured digraph on $v$ vertices in which there are exactly $\lambda$
edges of each colour in $C$ directed from $x$ to $y$ for any ordered pair $(x,y)$ of distinct
vertices. Let $\mathcal{H}$ be a family of edge-coloured digraphs whose edges are coloured with
colours from $C$. An $\mathcal{H}$-decomposition of an edge-coloured digraph $K$ is a
decomposition $\mathcal{D}$ of $K$ such that each edge-coloured digraph $G \in \mathcal{D}$ is
isomorphic to some graph in $\mathcal{H}$. For a graph $H \in \mathcal{H}$ and a vertex $x \in
V(H)$, we define $\tau(H,x)$ to be the $2r$-dimensional vector indexed by $C \times \{1,2\}$,
whose $(c,1)$ component is the number of edges coloured $c$ which are directed to $x$ and whose
$(c,2)$ component is the number of edges coloured $c$ which are directed from $x$. Let
$\alpha(\mathcal{H})$ denote the greatest common divisor of the integers $m$ such that
$m\vect{1_{2r}}$ is an integral linear combination of the vectors in $\{\tau(H,x):H \in
\mathcal{H}, x \in V(H)\}$. For a graph $H \in \mathcal{H}$, we define $\mu(H)$ to be the
$r$-dimensional vector indexed by $C$ whose $c$ component is the number of directed edges in $H$
which are coloured $c$. Let $\beta(\mathcal{H})$ denote the greatest common divisor of the
integers $m$ such that $m\vect{1_{r}}$ is an integral linear combination of the vectors in
$\{\mu(H):H \in \mathcal{H}\}$. We say that $\mathcal{H}$ is \emph{allowable} if $\vect{1_{r}}$
can be expressed as a linear combination of the vectors in $\{\mu(H):H \in \mathcal{H}\}$ with
strictly positive rational coefficients. Note that in \cite{LaWi} families with this last property
were called admissible but we rename it here to avoid confusion with our separate definition of
admissibility. Also note that in \cite{LaWi} this property is defined in a different way, but it
is also shown that the above definition is equivalent. The following result is given as Corollary
13.3 of \cite{LaWi}.

\begin{theorem}[\cite{LaWi}]\label{LamWilTheorem}
The graph $\lambda K^{(C)}_v$ admits a $\mathcal{H}$-decomposition for all sufficiently large
integers $v$ satisfying
\begin{itemize}
    \item
$\lambda(v-1) \equiv 0 \mod \alpha(\mathcal{H}))$; and
    \item
$\lambda v(v-1) \equiv 0 \mod \beta(\mathcal{H}))$;
\end{itemize}
provided that $\mathcal{H}$ is allowable.
\end{theorem}

Our main goal in this section will be to prove Lemmas \ref{ManyGroupGDD} and
\ref{ManyGroupGDDk=4}. Our proofs use techniques from \cite{LaWi} and closely follow the proof of
Theorem 8.1 of that paper, although we must be careful at times to ensure that the GDD we obtain
has the required blocking system. We will require the following well-known result (see \cite{Sc},
for example).

\begin{lemma}\label{IntLC}
Let $r$ be a positive integer. Given a set of $r$-dimensional rational vectors $U$, an
$r$-dimensional rational vector $\vect{c}$ can be written as an integral combination of the vectors
in $U$ if and only if, for every $r$-dimensional rational vector $\vect{y}$ such that the dot
product $\vect{y} \cdot \vect{u}$ is an integer for each $\vect{u} \in U$, the dot product
$\vect{y} \cdot \vect{c}$ is an integer.
\end{lemma}

\begin{lemma}\label{ManyGroupGDD}
Let $k$, $\lambda$ and $g$ be positive integers such that $k \geq 5$ and either $g=k-1$ or $g \geq
2k-2$. Then for each sufficiently large integer $t$ satisfying
\begin{itemize}
    \item[(i)]
$\lambda g(t-1) \equiv 0 \mod k-1)$; and
    \item[(ii)]
$\lambda g^2t(t-1) \equiv 0 \mod k(k-1))$;
\end{itemize}
there exists a $(k,\lambda)$-GDD of type $g^t$ which has a $2$-blocking system such that each set
of the blocking system intersects each group of the GDD in exactly $\lfloor\frac{g}{2}\rfloor$
points.
\end{lemma}

\begin{proof}[{\bf Proof}]
Throughout this proof, we will adopt the convention that if $\vect{u}$ is an $n$-dimensional
vector then, unless otherwise specified, $\vect{u}$ is indexed by $\{1,2,\ldots,n\}$ and component
$i$ of $\vect{u}$ is represented by $u_i$. For rational numbers $x$ and $y$ we shall use the
notation $x \equiv y$ to indicate that $x-y$ is an integer.

Let $G=\{1,2,\ldots,g\}$, let $G_1= \{1,2,\ldots,\lfloor\frac{g}{2}\rfloor\}$, let $G_2=
\{\lfloor\frac{g}{2}\rfloor+1,\lfloor\frac{g}{2}\rfloor+2,\ldots,2\lfloor\frac{g}{2}\rfloor\}$.
Let $R=G \times G$ be a set of colours. Let $F$ be the set of all $g$-dimensional integral vectors
$\vect{f}$ such that
\begin{itemize}
    \item
$f_i \geq 0$ for each $i \in \{1,2,\ldots,g\}$;
    \item
$f_1+f_2+\cdots+f_g=k$; and
    \item
$f_1 + f_2 + \cdots + f_{\lfloor\frac{g}{2}\rfloor} \geq 1$ and $f_{\lfloor\frac{g}{2}\rfloor+1}
+ f_{\lfloor\frac{g}{2}\rfloor+2} + \cdots + f_{2\lfloor\frac{g}{2}\rfloor} \geq 1$.
\end{itemize}
For each vector $\vect{f} \in F$ let $H_{\vect{f}}$ be the edge-coloured digraph with $k$ vertices
such that
\begin{itemize}
    \item
$V(H_{\vect{f}})$ has an ordered partition $(V_1,V_2,\ldots,V_g)$ such that $|V_i|=f_i$ for each
$i \in \{1,2,\ldots,g\}$; and
    \item
for any ordered pair of distinct vertices $(x,y)$ from $V(H_{\vect{f}})$ there is exactly one
directed edge from $x$ to $y$ and it has colour $(i,j)$, where $i$ and $j$ are the unique
elements of $G$ such that $x \in V_i$ and $y \in V_j$.
\end{itemize}
Let $\mathcal{H} = \{H_{\vect{f}}:\vect{f} \in F\}$.

It can be seen that an $\mathcal{H}$-decomposition of $\lambda K^{(R)}_t$ will yield a
$(k,\lambda)$-GDD of type $g^t$ (see \cite{LaWi} for details) and furthermore that this
$(k,\lambda)$-GDD will have a $2$-blocking system such that each set of the blocking system
intersects each group of the GDD in exactly $\lfloor\frac{g}{2}\rfloor$ points. (The two sets of
this blocking system will be formed by those points of the GDD which correspond to a colour in
$G_1$ and those points of the GDD which correspond to a colour in $G_2$, and the fact that $f_1 +
f_2 + \cdots + f_{\lfloor\frac{g}{2}\rfloor} \geq 1$ and $f_{\lfloor\frac{g}{2}\rfloor+1} +
f_{\lfloor\frac{g}{2}\rfloor+2} + \cdots + f_{2\lfloor\frac{g}{2}\rfloor} \geq 1$ for each
$\vect{f} \in F$ will guarantee that each block of the GDD intersects each set in the blocking
system in at at least one point.) So it suffices to show that for each sufficiently large integer
$t$ satisfying (i) and (ii), there is a decomposition of $\lambda K^{(R)}_t$ into copies of graphs
in $\mathcal{H}$.

Then by Theorem \ref{LamWilTheorem} it suffices to prove that, for each sufficiently large integer
$t$ satisfying (i) and (ii),
\begin{itemize}
    \item[(a)]
$\lambda t(t-1) \vect{1_{g^2}}$ is an integral linear combination of vectors in $\{\mu(H):H\in
\mathcal{H}\}$;
    \item[(b)]
$\lambda (t-1) \vect{1_{2g^2}}$ is an integral linear combination of vectors in
$\{\tau(H,y):H\in \mathcal{H} \hbox{ and } y \in V(H)\}$; and
    \item[(c)]
$\vect{1_{g^2}}$ can be expressed as a linear combination of the vectors in $\{\mu(H):H \in
\mathcal{H}\}$ with strictly positive rational coefficients.
\end{itemize}
This suffices because (a) guarantees that $\lambda t(t-1) \equiv 0 \mod \beta(\mathcal{H}))$, (b)
guarantees that $\lambda (t-1) \equiv 0 \mod \alpha(\mathcal{H}))$, and (c) guarantees that
$\mathcal{H}$ is allowable. Let $t$ be a positive integer satisfying (i) and (ii). We will prove
(a), (b) and (c) separately.

{\bf Proof of (a).} For each $\vect{f} \in F$, the $(i,i)$ component of $\mu(H_{\vect{f}})$ is
$f_i(f_i-1)$ for all $i \in \{1,2,\ldots,g\}$ and the $(i,j)$ component of $\mu(H_{\vect{f}})$ is
$f_if_j$ for all $i,j \in \{1,2,\ldots,g\}$ such that $i \neq j$. Thus by Lemma \ref{IntLC} it
suffices to prove that for any list of $g^2$ rational numbers $\{x_{ij}\}_{i,j \in
\{1,2,\ldots,g\}}$ satisfying
\begin{equation}\label{MainACong}
\sum_{i \neq j} f_if_jx_{ij} + \sum_i f_i(f_i-1)x_{ii} \equiv 0 \quad \hbox{for each } \vect{f} \in F,
\end{equation}
we have that
$$\sum_{i,j} \lambda t(t-1)x_{ij} \equiv 0.$$

Let $a$ and $b$ be distinct elements of $G$. Let $c$ be an element of $G \setminus \{a,b\}$ such
that $\{a,b,c\} \cap G_1 \neq \emptyset$ and $\{a,b,c\} \cap G_2 \neq \emptyset$. Let $\vect{f'}$
be the vector in $F$ such that $f'_a = k-2$, and $f'_b = f'_c = 1$. Let $\vect{f''}$ be the vector
in $F$ such that $f''_a=k-3$, $f''_b=2$ and $f''_c = 1$. Let $\vect{f'''}$ be the vector in $F$
such that $f'''_a = k-4$, $f'''_b = 3$ and $f'''_c = 1$. Subtracting twice the congruence implied
by (\ref{MainACong}) when $\vect{f}=\vect{f''}$ from the sum of the two congruences implied by
(\ref{MainACong}) when $\vect{f}=\vect{f'}$ and $\vect{f}=\vect{f'''}$ we see that
$$2x_{ab}+2x_{ba} \equiv 2x_{aa} + 2x_{bb}.$$
Thus,
\begin{equation}\label{CrossToPureACong}
2x_{ij}+2x_{ji} \equiv 2x_{ii} + 2x_{jj} \quad \hbox{ for all } i,j \in G,
\end{equation}
noting that the congruence is true trivially if $i=j$.

Let $a \in G \setminus \{1\}$ and let $b \in G_2 \setminus \{a\}$. If $k$ is odd, let
$\vect{f^\dag}$ be the vector in $F$ such that $f^\dag_1 = \frac{k-1}{2}$, $f^\dag_a =
\frac{k-3}{2}$ and $f^\dag_b = 2$, and let $\vect{f^\ddag}$ be the vector in $F$ such that
$f^\ddag_1 = \frac{k-3}{2}$, $f^\ddag_a = \frac{k-1}{2}$ and $f^\ddag_b = 2$. If $k$ is even, let
$\vect{f^\dag}$ be the vector in $F$ such that $f^\dag_1 = \frac{k}{2}$, $f^\dag_a =
\frac{k-2}{2}$ and $f^\dag_b = 1$, and let $\vect{f^\ddag}$ be the vector in $F$ such that
$f^\ddag_1 = \frac{k-2}{2}$, $f^\ddag_a = \frac{k}{2}$ and $f^\ddag_b = 1$. Subtracting the
congruence implied by (\ref{MainACong}) when $\vect{f}=\vect{f^\ddag}$ from the congruence implied
by (\ref{MainACong}) when $\vect{f}=\vect{f^\dag}$, doubling the resulting congruence if $k$ is
even, and then using (\ref{CrossToPureACong}) we see that
\begin{align*}
  (k-1)x_{11} &\equiv (k-1)x_{aa} \quad \hbox{if $k$ is odd, and} \\
  2(k-1)x_{11} &\equiv 2(k-1)x_{aa} \quad \hbox{if $k$ is even}.
\end{align*}
Thus,
\begin{align}\label{PureACongs}
  (k-1)x_{11} &\equiv (k-1)x_{ii} \quad \hbox{for all $i \in G$ if $k$ is odd, and} \nonumber \\
  2(k-1)x_{11} &\equiv 2(k-1)x_{ii} \quad \hbox{for all $i \in G$ if $k$ is even}.
\end{align}

Let $a \in G_2$ and let $\vect{f^*}$ be the vector in $F$ such that $f^*_1 = k-2$ and $f^*_a = 2$.
Using both (\ref{CrossToPureACong}) and (\ref{PureACongs}), it is easy to see from the congruence
implied by (\ref{MainACong}) when $\vect{f}=\vect{f^*}$ that $k(k-1)x_{11} \equiv 0$ and thus,
since $t$ satisfies (ii), we have
\begin{equation}\label{OnesACong}
\lambda g^2t(t-1)x_{11}\equiv 0.
\end{equation}

So, using (\ref{CrossToPureACong}), (\ref{PureACongs}) and (\ref{OnesACong}), noting that $\lambda
t(t-1)$ is a multiple of $2$, that $\lambda gt(t-1)$ is a multiple of $k-1$ if $k$ is odd (by
(i)), and that $\lambda gt(t-1)$ is a multiple of $2(k-1)$ if $k$ is even (by (i)), we have
$$\sum_{i,j} \lambda t(t-1)x_{ij} \equiv \sum_{i} \lambda gt(t-1)x_{ii} \equiv \lambda g^2t(t-1)x_{11} \equiv 0$$
as required.

{\bf Proof of (b).} Let $\vect{f}$ be a vector in $F$, let $x$ be a vertex of $H_{\vect{f}}$ and
let $\ell$ be the element of $G$ such that $x \in V_{\ell}$ (where $(V_1,V_2,\ldots,V_g)$ is the
ordered partition of $V(H_{\vect{f}})$ in the definition of $H_{\vect{f}}$). Then the
$((\ell,\ell),1)$ and $((\ell,\ell),2)$ components of $\tau(H_{\vect{f}},x)$ are $f_\ell-1$, the
$((i,\ell),1)$ and $((\ell,i),2)$ components of $\tau(H_{\vect{f}},x)$ are $f_i$ for all $i \in G
\setminus \{\ell\}$, and all the other components of $\tau(H_{\vect{f}},x)$ are $0$. Thus by Lemma
\ref{IntLC} it suffices to prove that for any list of $2g^2$ rational numbers
$\{x_{ij},y_{ij}\}_{i,j \in \{1,2,\ldots,g\}}$ satisfying
\begin{equation}\label{MainBCong}
(f_\ell-1)(x_{\ell\ell}+y_{\ell\ell}) + \sum_{i\neq \ell} f_i(x_{i\ell}+y_{\ell i}) \equiv 0 \quad \hbox{for each $\vect{f} \in F$ and $\ell \in G$ such that $f_\ell \geq 1$,}
\end{equation}
we have that
$$\sum_{i,j} \lambda (t-1)(x_{ij}+y_{ij}) \equiv 0.$$

Let $a$ and $b$ be distinct elements of $G$. Let $c$ be an element of $G \setminus \{a,b\}$ such
that $\{a,b,c\} \cap G_1 \neq \emptyset$ and $\{a,b,c\} \cap G_2 \neq \emptyset$. Let $\vect{f'}$
be the vector in $F$ such that $f'_a = k-2$, and $f'_b = f'_c = 1$. Let $\vect{f''}$ be the vector
in $F$ such that $f''_a=k-3$, $f''_b=2$ and $f''_c = 1$. Subtracting the congruence implied by
(\ref{MainBCong}) when $\vect{f}=\vect{f''}$ and $\ell=a$ from the congruence implied by
(\ref{MainBCong}) when $\vect{f}=\vect{f'}$ and $\ell=a$ we see that
$$x_{aa}+y_{aa} \equiv x_{ba} + y_{ab}.$$
Thus,
\begin{equation}\label{CrossToPureBCong}
x_{ii}+y_{ii} \equiv x_{ji} + y_{ij} \quad \hbox{ for all } i,j \in G.
\end{equation}
Using (\ref{CrossToPureBCong}), it is easy to see from the congruence implied by (\ref{MainBCong})
when $\vect{f}=\vect{f'}$ and $\ell=a$ that $(k-1)(x_{aa}+y_{aa}) \equiv 0$. Thus, we have
\begin{equation}\label{PureBCong}
(k-1)(x_{ii}+y_{ii}) \equiv 0 \quad \hbox{ for all  $i \in G$}.
\end{equation}

So, using (\ref{CrossToPureBCong}) and (\ref{PureBCong}), noting that $\lambda g(t-1)$ is a
multiple of $k-1$ (by (i)), we have
$$\sum_{i,j} \lambda (t-1)(x_{ij}+y_{ij}) \equiv \sum_{i} \lambda g(t-1)(x_{ii}+y_{ii}) \equiv 0$$
as required.

{\bf Proof of (c).} Let $\vect{p}= \sum_{\vect{f} \in F} \mu(H_{\vect{f}})$. Clearly $\epsilon
\vect{p}$ is a positive rational linear combination of the vectors in $\{\mu(H):H\in
\mathcal{H}\}$ for any positive rational $\epsilon$. Thus, it suffices to show that, for some
small positive rational number $\epsilon$, $\vect{1_{g^2}}-\epsilon\vect{p}$ is a non-negative
rational combination of the vectors in $\{\mu(H_{\vect{f}}):\vect{f} \in F\}$.

\emph{The case $g$ is odd.} Let $\ell$ be the integer such that $g=2\ell+1$. We will say that a
vector indexed by $G \times G$ is of type $(z_1,z_2,z_3,z_4,z_5)$, provided that for all $(i,j)
\in G \times G$ its $(i,j)$ component is $u_{ij}$ where
    $$ u_{ij} = \left\{
      \begin{array}{ll}
        z_1, & \hbox{if $i=j$ and $i \in G_1 \cup G_2$;} \\
        z_2, & \hbox{if $i \neq j$ and either $\{i,j\} \subseteq G_1$ or $\{i,j\} \subseteq G_2$;} \\
        z_3, & \hbox{if $i \neq j$ and either $(i,j) \in G_1 \times G_2$ or $(j,i) \in G_1 \times G_2$;} \\
        z_4, & \hbox{if $i=j=g$;} \\
        z_5, & \hbox{if $i \neq j$ and either $i=g$ or $j=g$.}
      \end{array}
    \right.$$
It can be seen that $\vect{p}$ is of type $(p_1,p_2,p_3,p_4,p_5)$ for some non-negative integers
$p_1,p_2,p_3,p_4,p_5$.

\emph{The case $g$ is odd and $g \geq 2k-1$.}  Note that $\ell \geq k-1 \geq 4$ in this case. We
define $F_1$, $F_2$, $F_3$, $F_4$ and $F_5$ to be subsets of $F$, as follows.
\begin{align*}
  F_1 &= \{\vect{f} \in F : f_i = k-1 \hbox{ for some } i \in G_1 \cup G_2\} \\
  F_2 &= \{\vect{f} \in F : f_i \leq 1 \hbox{ for all } i \in G, \{\textstyle{\sum_{i \in G_1}} f_i,\textstyle{\sum_{i \in G_2}} f_i\} = \{1,k-1\}\} \\
  F_3 &= \{\vect{f} \in F : f_i \leq 1 \hbox{ for all } i \in G, \{\textstyle{\sum_{i \in G_1}} f_i,\textstyle{\sum_{i \in G_2}} f_i\} = \{\lfloor\tfrac{k}{2}\rfloor,\lceil\tfrac{k}{2}\rceil\}\} \\
  F_4 &= \{\vect{f} \in F : f_g = k-2\} \\
  F_5 &= \{\vect{f} \in F : f_i \leq 1 \hbox{ for all } i \in G, \{\textstyle{\sum_{i \in G_1}} f_i,\textstyle{\sum_{i \in G_2}} f_i\} = \{\lfloor\tfrac{k-1}{2}\rfloor,\lceil\tfrac{k-1}{2}\rceil\}\}
\end{align*}

Through routine but tedious counting it can be calculated that $\frac{1}{|F_i|}\sum_{\vect{f} \in
F_i} \mu(H_{\vect{f}})$ is of type $\vect{a}$ if $i=1$, $\vect{b}$ if $i=2$, $\vect{c}$ if $i=3$,
$\vect{d}$ if $i=4$, and $\vect{e}$ if $i=5$, where
\begin{align*}
  \vect{a} &= (\tfrac{(k-1)(k-2)}{2\ell},0,\tfrac{k-1}{\ell^2},0,0), \\
  \vect{b} &= (0,\tfrac{(k-1)(k-2)}{2\ell(\ell-1)},\tfrac{k-1}{\ell^2},0,0), \\
  \vect{c} &= (0,\tfrac{1}{2\ell(\ell-1)}(\lceil\tfrac{k}{2}\rceil\lceil\tfrac{k-2}{2}\rceil+\lfloor\tfrac{k}{2}\rfloor\lfloor\tfrac{k-2}{2}\rfloor),\tfrac{1}{\ell^2}(\lceil\tfrac{k}{2}\rceil\lfloor\tfrac{k}{2}\rfloor),0,0), \\
  \vect{d} &= (0,0,\tfrac{1}{\ell^2},(k-2)(k-3),\tfrac{k-2}{\ell}), \\
  \vect{e} &= (0,\tfrac{1}{2\ell(\ell-1)}(\lceil\tfrac{k-1}{2}\rceil\lceil\tfrac{k-3}{2}\rceil+\lfloor\tfrac{k-1}{2}\rfloor\lfloor\tfrac{k-3}{2}\rfloor),\tfrac{1}{\ell^2}(\lceil\tfrac{k-1}{2}\rceil\lfloor\tfrac{k-1}{2}\rfloor),0,\tfrac{k-1}{2\ell}).
\end{align*}
To show that $\vect{1_{g^2}}-\epsilon\vect{p}$ is a non-negative rational linear combination of
vectors in $\{\mu(H):H\in \mathcal{H}\}$ it suffices to show that $(1-\epsilon p_1,1-\epsilon
p_2,1-\epsilon p_3,1-\epsilon p_4,1-\epsilon p_5)$ is a non-negative rational combination of
$\vect{a}$, $\vect{b}$, $\vect{c}$, $\vect{d}$ and $\vect{e}$.

Simple calculations give us that
$$\tfrac{2\ell(1-\epsilon p_1)}{(k-1)(k-2)}\vect{a}+\tfrac{(1-\epsilon p_4)}{(k-2)(k-3)}\vect{d}+\tfrac{2\ell(k-3)(1-\epsilon p_5)-2(1-\epsilon p_4)}{(k-1)(k-3)}\vect{e}=(1-\epsilon p_1,y_2,y_3,1-\epsilon p_4,1-\epsilon p_5) \quad\hbox{where}$$
\begin{align*}
  \lim_{\epsilon \rightarrow 0} y_2 &= x_2, & x_2 &= \tfrac{\ell(k-3)-1}{\ell(\ell-1)(k-1)(k-3)}(\lceil\tfrac{k-1}{2}\rceil\lceil\tfrac{k-3}{2}\rceil+\lfloor\tfrac{k-1}{2}\rfloor\lfloor\tfrac{k-3}{2}\rfloor) \quad \hbox{and} \\
  \lim_{\epsilon \rightarrow 0} y_3 &= x_3, & x_3 &= \tfrac{1}{\ell^2(k-2)(k-3)}+\tfrac{2}{\ell(k-2)}+\tfrac{2(\ell(k-3)-1)}{\ell^2(k-1)(k-3)}(\lceil\tfrac{k-1}{2}\rceil\lfloor\tfrac{k-1}{2}\rfloor).
\end{align*}
So it suffices to show that $(0,1-\epsilon p_2-y_2,1-\epsilon p_3-y_3,0,0)$ is a non-negative
rational combination of $\vect{b}$ and $\vect{c}$ and hence, since $b_2$, $1-\epsilon p_2-y_2$ and
$c_2$ are all positive for sufficiently small positive values of $\epsilon$, it suffices to show
that
$$\frac{b_3}{b_2} \leq \frac{1-\epsilon p_3-y_3}{1-\epsilon p_2-y_2} \leq \frac{c_3}{c_2}.$$
Now $\lim_{\epsilon \rightarrow 0} 1-\epsilon p_2-y_2=1-x_2$ and $\lim_{\epsilon \rightarrow 0}
1-\epsilon p_3-y_3=1-x_3$, so it suffices to show that
$$\frac{b_3}{b_2} < \frac{1-x_3}{1-x_2} < \frac{c_3}{c_2}.$$
Let $\Delta_1=b_2(1-x_3)-b_3(1-x_2)$ and $\Delta_2=c_3(1-x_2)-c_2(1-x_3)$. We will show that
$\Delta_1$ and $\Delta_2$ are both positive. Substituting in for $b_2$, $b_3$, $c_2$, $c_3$, $x_2$
and $x_3$ and simplifying yields that when $k$ is even
\begin{align*}
  \Delta_1 &= \tfrac{1}{4\ell^3(\ell-1)(k-3)}(2\ell(\ell-k+1)(k-1)(k-3)(k-4)+k(\ell(k-3)+1)(k^2-6k+6)+(4k-6)) \hbox{ and} \\
  \Delta_2 &= \tfrac{k}{4\ell^3(\ell-1)(k-3)}(2\ell(\ell-k+1)(k-3)+k\ell(k-3)+1),
\end{align*}
and when $k$ is odd
\begin{align*}
  \Delta_1 &= \tfrac{k-1}{4\ell^3(\ell-1)}(2\ell(\ell-k+1)(k-4)+k\ell(k-5)+k-2)
\hbox{ and} \\
  \Delta_2 &= \tfrac{k-1}{4\ell^3(\ell-1)(k-2)}(2\ell(\ell-k+1)(k-2)+k\ell(k-1)-1).
\end{align*}
Given that $k \geq 5$ and that $\ell \geq k-1$, it is now routine to confirm that $\Delta_1$ and
$\Delta_2$ are positive, as required.

\emph{The case $g$ is odd and $g = k-1$.} Note that $k \geq 6$ and $\ell = \frac{k-2}{2} \geq 2$
in this case. We first deal with the case where $\ell \geq 4$. Define $F_1$, $F_2$, $F_3$, $F_4$
and $F_5$ to be subsets of $F$, as follows.
\begin{align*}
  F_1 =& \{\vect{f} \in F : f_i = 2\ell+1 \hbox{ for some } i \in G_1 \cup G_2\} \\
  F_2 =& \{\vect{f} \in F : f_i = f_j = \ell+1 \hbox{ for some } i,j \in G\} \\
  F_3 =& \{\vect{f} \in F : \hbox{for some } i,j,m \hbox{ such that } \{\{i,j,m\} \cap G_1,\{i,j,m\} \cap G_2\} = \{\{i\},\{j,m\}\}, \\[-0.2cm]
  &(f_i,f_j,f_m)=(0,2,3) \hbox{ and } f_h=1 \hbox{ for each } h \in (G_1 \cup G_2) \setminus \{i,j,m\}\} \\ 
  F_4 =& \{\vect{f} \in F : f_g = 3, f_i \leq 1 \hbox{ for all } i \in G_1 \cup G_2\} \\
  F_5 =& \{\vect{f} \in F : f_g = 1, f_i \in \{1,2\} \hbox{ for all } i \in G_1 \cup G_2\}
\end{align*}

Through routine but tedious counting it can be calculated that $\frac{1}{|F_i|}\sum_{\vect{f} \in
F_i} \mu(H_{\vect{f}})$ is of type $\vect{a}$ if $i=1$, $\vect{b}$ if $i=2$, $\vect{c}$ if $i=3$,
$\vect{d}$ if $i=4$, and $\vect{e}$ if $i=5$, where
\begin{align*}
  \vect{a} &= (2\ell+1,0,\tfrac{2\ell+1}{\ell^2},0,0), \\
  \vect{b} &= (\ell+1,0,\tfrac{(\ell+1)^2}{\ell^2},0,0), \\
  \vect{c} &= (\tfrac{4}{\ell},\tfrac{\ell+1}{\ell-1},\tfrac{(\ell-1)(\ell+3)}{\ell^2},0,0), \\
  \vect{d} &= (0,\tfrac{\ell-1}{\ell},\tfrac{\ell-1}{\ell},6,\tfrac{6\ell-3}{2\ell}), \\
  \vect{e} &= (\tfrac{1}{\ell},\tfrac{\ell+1}{\ell},\tfrac{\ell+1}{\ell},0,\tfrac{2\ell+1}{2\ell}).
\end{align*}
Again, to show that $\vect{1_{g^2}}-\epsilon\vect{p}$ is a non-negative rational linear
combination of vectors in $\{\mu(H):H\in \mathcal{H}\}$ it suffices to show that $(1-\epsilon
p_1,1-\epsilon p_2,1-\epsilon p_3,1-\epsilon p_4,1-\epsilon p_5)$ is a non-negative rational
combination of $\vect{a}$, $\vect{b}$, $\vect{c}$, $\vect{d}$ and $\vect{e}$.

Simple calculations give us that
$$\tfrac{(\ell-1)((1-\epsilon p_4)(2\ell^2+2\ell-1)-6(1-\epsilon p_5)(\ell^2+\ell)+3(1-\epsilon p_2)(2\ell^2+\ell))}{3\ell(\ell+1)(2\ell+1)}\vect{c}+\tfrac{(1-\epsilon p_4)}{6}\vect{d}+\tfrac{4\ell(1-\epsilon p_5)-(2\ell-1)(1-\epsilon p_4)}{4\ell+2}\vect{e}$$
is equal to $(y_1,1-\epsilon p_2,y_3,1-\epsilon p_4,1-\epsilon p_5)$ where
\begin{align*}
  \lim_{\epsilon \rightarrow 0}y_1 &= x_1, & x_1 &= \tfrac{11\ell^2-13\ell+8}{6\ell^2(\ell+1)} \quad \hbox{and} \\
  \lim_{\epsilon \rightarrow 0}y_3 &= x_3, & x_3 &= \tfrac{3\ell^4+3\ell^3-5\ell^2+8\ell-3}{3\ell^3(\ell+1)}.
\end{align*}
So it suffices to show that $(1-\epsilon p_1-y_1,0,1-\epsilon p_3-y_3,0,0)$ is a non-negative
rational combination of $\vect{a}$ and $\vect{b}$. For $\ell \geq 4$, this can be shown in a
similar manner to that used in the case where $g \geq 2k-1$.

The case $\ell \in \{2,3\}$ can be dealt with similarly, except that we also define
    $$F'_5 = \{\vect{f} \in F : \{(\textstyle{\sum_{i \in G_1}} f_i,|\{i \in G_1:f_i=1\}|),(\textstyle{\sum_{i \in G_2}f_i,|\{i \in G_2:f_i=1\}|)} \} = \{(2\ell,\ell-1),(1,1)\}\},$$
note that $\frac{1}{|F'_5|}\sum_{\vect{f} \in F'_5} \mu(H_{\vect{f}})$ is of type $\vect{e'}$ where
$\vect{e'} = (\tfrac{\ell+1}{2},\tfrac{3}{2},\tfrac{2}{\ell},0,\tfrac{2\ell+1}{2\ell})$, and
include \linebreak $(\tfrac{4\ell(1-\epsilon p_5)-(2\ell-1)(1-\epsilon
p_4)}{4\ell+2})(\frac{3}{4}\vect{e}+\frac{1}{4}\vect{e'})$ in our linear combination rather than
$\tfrac{4\ell(1-\epsilon p_5)-(2\ell-1)(1-\epsilon p_4)}{4\ell+2}\vect{e}$.

\emph{The case $g$ is even.} The arguments in this case are similar to, but less complicated than,
those made in the case where $g$ is odd. Let $\ell$ be the integer such that $g=2\ell$. We will
say that a vector indexed by $G \times G$ is of type $(z_1,z_2,z_3)$, provided that for all $(i,j)
\in G \times G$ its $(i,j)$ component is $u_{ij}$ where
    $$ u_{ij} = \left\{
      \begin{array}{ll}
        z_1, & \hbox{if $i=j$ and $i \in G_1 \cup G_2$;} \\
        z_2, & \hbox{if $i \neq j$ and either $\{i,j\} \subseteq G_1$ or $\{i,j\} \subseteq G_2$;} \\
        z_3, & \hbox{if $i \neq j$ and either $(i,j) \in G_1 \times G_2$ or $(j,i) \in G_1 \times G_2$.} \\
      \end{array}
    \right.$$
It can be seen that $\vect{p}$ is of type $(p_1,p_2,p_3)$ for some non-negative integers
$p_1,p_2,p_3$.

\emph{The case $g$ is even and $g \geq 2k-2$.}  Note that $\ell \geq k-1 \geq 4$ in this case. We
define $F_1$, $F_2$ and $F_3$ as follows.
\begin{align*}
  F_1 &= \{\vect{f} \in F : f_i = k-1 \hbox{ for some } i \in G \} \\
  F_2 &= \{\vect{f} \in F : f_i \leq 1 \hbox{ for all } i \in G, \{\textstyle{\sum_{i \in G_1}} f_i,\textstyle{\sum_{i \in G_2}} f_i\} = \{1,k-1\}\} \\
  F_3 &= \{\vect{f} \in F : f_i \leq 1 \hbox{ for all } i \in G, \{\textstyle{\sum_{i \in G_1}} f_i,\textstyle{\sum_{i \in G_2}} f_i\} = \{\lfloor\tfrac{k}{2}\rfloor,\lceil\tfrac{k}{2}\rceil\}\}
\end{align*}
The proof proceeds along similar lines to the cases where $g$ is odd (though it is less
complicated).

\emph{The case $g$ is even and $g =k-1$.} Note that $\ell = \frac{k-1}{2} \geq 2$ in this case. We
define $F_1$, $F_2$ and $F_3$ as follows.
\begin{align*}
  F_1 =& \{\vect{f} \in F : f_i = 2\ell \hbox{ for some } i \in G\} \\
  F_2 =& \{\vect{f} \in F : f_i = \ell, f_j = \ell+1 \hbox{ for some } i,j \in G\} \\
  F_3 =& \{\vect{f} \in F : \hbox{for some } i,j,m \hbox{ such that } \{\{i,j,m\} \cap G_1,\{i,j,m\} \cap G_2\} = \{\{i\},\{j,m\}\}, \\[-0.2cm]
  &(f_i,f_j,f_m)=(0,2,2) \hbox{ and } f_h=1 \hbox{ for each } h \in (G_1 \cup G_2) \setminus \{i,j,m\}\}
\end{align*}
When $\ell \geq 4$, the proof proceeds along similar lines to the cases where $g$ is odd (though
it is less complicated). When $\ell \in \{2,3\}$, we proceed similarly except that we also make
use of
$$F'_3 = \{\vect{f} \in F : f_i \in \{1,2\} \hbox{ for all } i \in G\}.$$
\end{proof}

With more work, the restriction that either $g=k-1$ or $g \geq 2k-2$ in Lemma \ref{ManyGroupGDD}
could certainly be loosened. The above result suffices for our purposes here, however.

\begin{lemma}\label{ManyGroupGDDk=4}
Let $\lambda$ and $g$ be positive integers such that $g \geq 6$ and  $g$ is even. Then for each
sufficiently large integer $t$ satisfying $\lambda g(t-1) \equiv 0 \mod 12)$; there exists a
$(4,\lambda)$-GDD of type $g^t$ which has a $2$-blocking system such that each set of the blocking
system intersects each group of the GDD in exactly $\frac{g}{2}$ points.
\end{lemma}

\begin{proof}[{\bf Proof}]
The proof proceeds along similar lines to the proof of Lemma \ref{ManyGroupGDD}, so we highlight
only the points of difference. Let $F$ be the set of all $g$-dimensional integral vectors
$\vect{f}$ such that
\begin{itemize}
    \item
$f_i \geq 0$ for each $i \in \{1,2,\ldots,g\}$;
    \item
$f_1+f_2+\cdots+f_g=4$; and
    \item
$\{f_1 + f_2 + \cdots + f_{\frac{g}{2}},f_{\frac{g}{2}+1} + f_{\frac{g}{2}+2} + \cdots + f_{g}\}
= \{1,3\}$.
\end{itemize}

{\bf Proof of (a).} Let $a$, $b$, $c$ and $d$ be distinct elements of $G$ such that either $a,b,c
\in G_1$ and $d \in G_2$ or $d \in G_1$ and $a,b,c \in G_2$. Let $\vect{f'}$, $\vect{f''}$,
$\vect{f'''}$, $\vect{f^\dag}$ and $\vect{f^*}$ be the vectors in $F$ such that $(f'_a,f'_b,f'_d)
= (2,1,1)$, $(f''_a,f''_b,f''_d) = (1,2,1)$, $(f'''_b,f'''_d) = (3,1)$, $(f^\dag_b,f^\dag_d) =
(1,3)$ and $(f^*_a,f^*_b,f^*_c,f^*_d)=(1,1,1,1)$. The congruence implied by (\ref{MainACong}) when
$\vect{f}=\vect{f'''}$ yields $3x_{bd}+3x_{db} \equiv -6x_{bb}$. Thus,
\begin{equation}\label{ExtCrossToPureACong}
3x_{ij}+3x_{ji} \equiv -6x_{jj} \quad \hbox{ for all } i,j \in G, \hbox{ such that } |\{i,j\} \cap G_1|=1.
\end{equation}
Subtracting the congruence implied by (\ref{MainACong}) when $\vect{f}=\vect{f^\dag}$ from the the
congruence implied by (\ref{MainACong}) when $\vect{f}=\vect{f'''}$ we see that $6x_{bb} -6x_{dd}
\equiv 0$. Thus,
\begin{equation}\label{PureACongk=4}
6x_{ii} -6x_{jj} \equiv 0 \quad \hbox{ for all } i,j \in G \hbox{ such that } |\{i,j\} \cap G_1|=1.
\end{equation}
Subtracting twice the congruence implied by (\ref{MainACong}) when $\vect{f}=\vect{f''}$ from the
sum of the two congruences implied by (\ref{MainACong}) when $\vect{f}=\vect{f'}$ and
$\vect{f}=\vect{f'''}$ we see that $2x_{ab}+2x_{ba} \equiv 2x_{aa} + 2x_{bb}$. Thus,
\begin{equation}\label{IntCrossToPureACong}
2x_{ij}+2x_{ji} \equiv 2x_{ii} + 2x_{jj} \quad \hbox{ for all distinct } i,j \in G \hbox{ such that } |\{i,j\} \cap G_1|\in \{0,2\}.
\end{equation}
Finally, applying (\ref{IntCrossToPureACong}) to twice the congruence implied by (\ref{MainACong})
when $\vect{f}=\vect{f^*}$ yields $2x_{ad}+2x_{da}+2x_{bd}+2x_{db}+2x_{cd}+2x_{dc} \equiv
-4x_{aa}-4x_{bb}-4x_{cc}$. Thus,
\begin{multline}\label{Mess}
2x_{im}+2x_{mi}+2x_{jm}+2x_{mj}+2x_{km}+2x_{mk} \equiv -4x_{ii}-4x_{jj}-4x_{kk} \\
\hbox{ for all distinct } i,j,k,m \in G \hbox{ such that } \{i,j,k,m\} \cap G_1\in \{\{i,j,k\},\{m\}\}.
\end{multline}

Combining these facts we see that
\begin{align*}
  \sum_{i,j} \lambda t(t-1)x_{ij} &\equiv \sum_{i} \lambda \tfrac{g}{2}t(t-1)x_{ii} + \sum_{i \in G_1,j \in G_2} \lambda t(t-1)(x_{ij}+x_{ji}) \quad (\hbox{using } (\ref {IntCrossToPureACong})) \\
  &\equiv \sum_{i \in G_1} \lambda \tfrac{g}{2}t(t-1)x_{ii} - \sum_{j \in G_2} \lambda \tfrac{g}{2}t(t-1)x_{jj} \quad (\hbox{using } (\ref{ExtCrossToPureACong}) \hbox{ or } (\ref {Mess})) \\
  &\equiv 0 \quad (\hbox{using } (\ref{PureACongk=4})).
\end{align*}
In the above we apply (\ref{ExtCrossToPureACong}) when $3$ does not divide $g$ and hence $3$
divides $\lambda(t-1)$ by our hypotheses and we apply (\ref{Mess}) when $3$ divides $g$ and hence
$3$ divides $\frac{g}{2}$.

{\bf Proof of (b).} Let $a$, $b$ and $c$ be distinct elements of $G$ such that either $a,b \in
G_1$ and $c \in G_2$ or $c \in G_1$ and $a,b \in G_2$. Let $\vect{f'}$, $\vect{f''}$ and
$\vect{f'''}$ be the vectors in $F$ such that $(f'_a,f'_b,f'_c) = (2,1,1)$, $(f''_a,f''_b,f''_c) =
(1,2,1)$ and $(f'''_b,f'''_c) = (3,1)$. Subtracting the congruence implied by (\ref{MainBCong})
when $\vect{f}=\vect{f''}$ and $\ell=a$ from the congruence implied by (\ref{MainBCong}) when
$\vect{f}=\vect{f'}$ and $\ell=a$ we see that $x_{ba} + y_{ab} \equiv x_{aa}+y_{aa}$. Thus,
\begin{equation}\label{IntCrossToPureBCong}
x_{ji} + y_{ij} \equiv x_{ii}+y_{ii} \quad \hbox{ for all } i,j \in G \hbox{ such that } |\{i,j\} \cap G_1|\in \{0,2\}.
\end{equation}
The congruence implied by (\ref{MainBCong}) when $\vect{f}=\vect{f'''}$ and $\ell=b$ yields
$x_{cb}+y_{bc} \equiv -2x_{bb}-2y_{bb}$. Thus,
\begin{equation}\label{ExtCrossToPureBCong}
x_{ji}+y_{ij} \equiv -2x_{ii}-2y_{ii} \quad \hbox{ for all } i,j \in G, \hbox{ such that } |\{i,j\} \cap G_1|=1.
\end{equation}
Using (\ref{ExtCrossToPureBCong}), it is easy to see from the congruence implied by
(\ref{MainBCong}) when $\vect{f}=\vect{f'''}$ and $\ell=c$ that $6x_{cc}+6y_{cc} \equiv 0$. Thus,
\begin{equation}\label{PureBCongk=4}
6x_{ii}+6y_{ii} \equiv 0 \quad \hbox{ for all } i \in G.
\end{equation}

Combining these facts we see that
\begin{align*}
  \sum_{i,j} \lambda (t-1)(x_{ij}+y_{ij}) &\equiv \sum_{i} \lambda \tfrac{g}{2}(t-1)(x_{ii}+y_{ii}) + \sum_{i \in G_1,j \in G_2} \lambda (t-1)(x_{ij}+x_{ji}) \quad (\hbox{using } (\ref {IntCrossToPureBCong})) \\
  &\equiv -\sum_{i} \lambda \tfrac{g}{2}(t-1)(x_{ii}+y_{ii}) \quad (\hbox{using } (\ref {ExtCrossToPureBCong})) \\
  &\equiv 0 \quad (\hbox{using } (\ref{PureBCongk=4})).
\end{align*}
In the above, applying (\ref{PureBCongk=4}) requires noting that $\lambda \tfrac{g}{2}(t-1)$ is a
multiple of $6$ which follows from the hypotheses of the lemma.

{\bf Proof of (c).} Let $\ell$ be the integer such that $g=2\ell$. We define $F_1$, $F_2$ and
$F_3$ as follows.
\begin{align*}
  F_1 &= \{\vect{f} \in F : f_i = 3 \hbox{ for some } i \in G \} \\
  F_2 &= \{\vect{f} \in F : f_i = 2 \hbox{ for some } i \in G \} \\
  F_3 &= \{\vect{f} \in F : f_i \leq 1 \hbox{ for all } i \in G\}
\end{align*}
Routine calculation yields that
    $$\tfrac{\ell}{6|F_1|}\sum_{\vect{f} \in F_1} \mu(H_{\vect{f}})+
    \tfrac{\ell}{2|F_2|}\sum_{\vect{f} \in F_2} \mu(H_{\vect{f}})+
    \tfrac{\ell(\ell-2)}{3|F_3|}\sum_{\vect{f} \in F_3} \mu(H_{\vect{f}}) = \vect{1_{g^2}}.$$
Furthermore, $F=F_1 \cup F_2 \cup F_3$, which means that the left hand side of the above equation
is a linear combination of the vectors in $\{\mu(H):H \in \mathcal{H}\}$ with strictly positive
rational coefficients.
\end{proof}

Lemma \ref{ManyGroupGDDk=4} is an analogue of Lemma \ref{ManyGroupGDD} with the additional
restrictions that $g$ is even and $\lambda g (t-1) \equiv 0 \mod 4)$. To see that these conditions
are necessary in the case $k=4$, suppose there exists a $(4,\lambda)$-GDD of type $g^t$ which has
a $2$-blocking system $\{S_1,S_2\}$ such that $S_1$ and $S_2$ each intersect each group of the GDD
in exactly $\lfloor\frac{g}{2}\rfloor$ points. Every block of this GDD contains at least three
pairs of points which intersect $S_1$ in exactly one point, and thus at least half of the pairs of
points which appear in blocks of the GDD must intersect $S_1$ in exactly one point. It follows
that $g$ is even and that every block in the GDD intersects $S_1$ in one or three points. Thus,
for a fixed point $x \in S_1$ there are $\lambda \frac{g}{2} (t-1)$ pairs of points including $x$
and a point in $S_1$ in a different group to $x$ and each block of the GDD contains zero or two of
these pairs. It follows that $\lambda g (t-1) \equiv 0 \mod 4)$.

By combining Lemmas \ref{ManyGroupGDD} and \ref{ManyGroupGDDk=4} with some standard ``group
filling'' constructions, we can obtain the following two results.

\begin{lemma}\label{BasicConstructionsInfinity}
Let $y$, $k$ and $\lambda$ be positive integers such that $k \geq 4$, either $y=k$ or $y \geq
2k-1$, and if $k=4$ then $y$ is odd. If there exists a $(y,k,\lambda)$-BIBD which has a
$(\lfloor\frac{y-1}{2}\rfloor,\lfloor\frac{y-1}{2}\rfloor)$-blocking system, then, for each
sufficiently large integer $x$ such that $x(y-1)+1$ is $(k,\lambda)$-admissible,
\begin{itemize}
    \item[(a)]
there exists an $(x(y-1)+1,k,\lambda)$-BIBD which has an
$(x\lfloor\frac{y-1}{2}\rfloor,x\lfloor\frac{y-1}{2}\rfloor)$-blocking system; and
    \item[(b)]
there exists a $(k,\lambda)$-GDD of type $y^11^{(x-1)(y-1)}$ which has an
$(x\lfloor\frac{y-1}{2}\rfloor,x\lfloor\frac{y-1}{2}\rfloor)$-blocking system such that each
set of the blocking system intersects the group of size $y$ in exactly
$\lfloor\frac{y-1}{2}\rfloor$ points.
\end{itemize}
\end{lemma}

\begin{proof}[{\bf Proof}] Since $y$ is $(k,\lambda)$-admissible, it is easy to check that, for a sufficiently large
integer $x$ such that $x(y-1)+1$ is $(k,\lambda)$-admissible, Lemma \ref{ManyGroupGDD} or Lemma
\ref{ManyGroupGDDk=4} implies that there exists a $(k,\lambda)$-GDD $(V,\mathcal{G},\mathcal{A})$
of type $(y-1)^x$ which has a $2$-blocking system $\{S_1,S_2\}$ such that $|S_1 \cap G| = |S_2
\cap G| = \lfloor\frac{y-1}{2}\rfloor$ for each $G \in \mathcal{G}$ (note that if $k=4$ then $y$
is odd and so the fact that $y$ is $(4,\lambda)$-admissible implies $\lambda(x-1)(y-1) \equiv 0
\mod 12)$). Now let $\infty$ be a point not in $V$, let $G^* \in \mathcal{G}$ and for each $G \in
\mathcal{G}$ let $\mathcal{A}_G$ be a collection of blocks such that $(G \cup
\{\infty\},\mathcal{A}_G)$ is a $(y,k,\lambda)$-BIBD for which $\{S_1 \cap G,S_2 \cap G\}$ is a
blocking system. Let
$$\mathcal{B}=\mathcal{A} \cup \bigcup_{G \in \mathcal{G}} \mathcal{A}_G.$$
Then $(V \cup \{\infty\},\mathcal{B})$ is the required BIBD, $(V \cup \{\infty\},\{G^* \cup
\{\infty\}\} \cup \{\{z\}:z \in V \setminus G^*\},\mathcal{B}\setminus\mathcal{A}_{G^*})$ is the
required GDD, and in both cases $\{S_1,S_2\}$ is the required blocking system.
\end{proof}

\begin{lemma}\label{BasicConstructionsNoInfinity}
Let $y$, $k$ and $\lambda$ be positive integers such that $k \geq 4$, $y \geq 2k-2$ and if $k=4$
then $y$ is even. If there exists a $(y,k,\lambda)$-BIBD which has a
$(\lfloor\frac{y}{2}\rfloor,\lfloor\frac{y}{2}\rfloor)$-blocking system then, for each sufficiently
large integer $x$ such that $xy$ is $(k,\lambda)$-admissible,
\begin{itemize}
    \item[(a)]
there exists an $(xy,k,\lambda)$-BIBD which has an
$(x\lfloor\frac{y}{2}\rfloor,x\lfloor\frac{y}{2}\rfloor)$-blocking system; and
    \item[(b)]
there exists a $(k,\lambda)$-GDD of type $y^11^{(x-1)y}$ which has an
$(x\lfloor\frac{y}{2}\rfloor,x\lfloor\frac{y}{2}\rfloor)$-blocking system such that each set of
the blocking system intersects the group of size $y$ in exactly $\lfloor\frac{y}{2}\rfloor$
points.
\end{itemize}
\end{lemma}

\begin{proof}[{\bf Proof}]
This is proved very similarly to Lemma \ref{BasicConstructionsInfinity}, except that we take
a base GDD of type $y^x$ and we do not add the point $\infty$.
\end{proof}

\section{Examples of $2$-chromatic BIBDs}

In this section we will use Lemmas \ref{BasicConstructionsInfinity} and
\ref{BasicConstructionsNoInfinity} to find various examples of $2$-chromatic BIBDs. In Lemma
\ref{2Chromatic1ModThing} we establish, for all positive integers $k$ and $\lambda$ with $k \geq
5$, the asymptotic existence of $2$-chromatic BIBDs with block size $k$, index $\lambda$ and order
congruent to $1$ modulo $k-1$. We then construct, for all positive integers $k$ and $\lambda$ with
$k \geq 4$, $2$-chromatic BIBDs with block size $k$ and index $\lambda$ whose orders fall in each
admissible congruence class modulo $k(k-1)$. This is accomplished in Lemma
\ref{2ChromaticInEachCongClass} for $k \geq 5$ and in Lemma \ref{2ChromaticBlockSize4} for $k=4$.


\begin{lemma}\label{2Chromatic1ModThing}
Let $k$ and $\lambda$ be positive integers such that $k \geq 5$. For each sufficiently large
$(k,\lambda)$-admissible integer $v$ such that $v \equiv 1 \mod k-1)$, there exists a
$(v,k,\lambda)$-BIBD which has a
$(\lfloor\frac{v-1}{2}\rfloor,\lfloor\frac{v-1}{2}\rfloor)$-blocking system.
\end{lemma}

\begin{proof}[{\bf Proof}]
The trivial $(k,k,\lambda)$-BIBD obviously has a
$(\lfloor\frac{k-1}{2}\rfloor,\lfloor\frac{k-1}{2}\rfloor)$-blocking system. Thus, by Lemma
\ref{BasicConstructionsInfinity} (a), it can be seen that, for each sufficiently large integer $x$
such that $x(k-1)+1$ is $(k,\lambda)$-admissible, there exists an $(x(k-1)+1,k,\lambda)$-BIBD with
an $(x\lfloor\frac{k-1}{2}\rfloor,x\lfloor\frac{k-1}{2}\rfloor)$-blocking system. Since
$x\lfloor\frac{k-1}{2}\rfloor \leq \lfloor\frac{x(k-1)}{2}\rfloor$ for all positive integers $x$,
the proof is complete.
\end{proof}

Note that, for any integer $k \geq 5$, the above lemma implies that a $2$-chromatic $(v,k,1)$-BIBD
exists for each sufficiently large $(k,1)$-admissible integer $v$.

\begin{lemma}\label{2ChromaticInEachCongClass}
Let $k$ and $\lambda$ be positive integers such that $k \geq 5$. For each $(k,\lambda)$-admissible
integer $m \in \{0,1,\ldots, k(k-1)-1\}$, there is a positive integer $z$ such that $z \geq 2k-1$,
$z \equiv m \mod k(k-1))$, and there exists a $(z,k,\lambda)$-BIBD which has a
$(\lfloor\frac{z-1}{2}\rfloor,\lfloor\frac{z-1}{2}\rfloor)$-blocking system.
\end{lemma}

\begin{proof}[{\bf Proof}] Let $m$ be a $(k,\lambda)$-admissible element of $\{0,1,\ldots,k(k-1)-1\}$. By Lemma
\ref{2Chromatic1ModThing} there exists an integer $y' \equiv 1 \mod k(k-1))$ such that $y' \geq
2k-1$ and there is a $(y',k,\lambda)$-BIBD which has a $(\frac{y'-1}{2},\frac{y'-1}{2})$-blocking
system (note that $y'$ is odd since $k(k-1)$ is even). Thus, since any positive integer congruent
to $m$ modulo $k(k-1)$ is itself $(k,\lambda)$-admissible, by Lemma
\ref{BasicConstructionsNoInfinity} (a) it can be seen that there is a positive integer $x$ such
that $x \equiv m \mod k(k-1))$ and there exists an $(xy',k,\lambda)$-BIBD which has an
$(\frac{x(y'-1)}{2},\frac{x(y'-1)}{2})$-blocking system. Since $xy' \geq 2k-1$, $xy' \equiv m \mod
k(k-1))$ and $\frac{x(y'-1)}{2} \leq \lfloor\frac{xy'-1}{2}\rfloor$, the proof is complete.
\end{proof}

\begin{lemma}\label{2ChromaticBlockSize4}
Let $z$ and $\lambda$ be positive integers such that $z \in \{6,7,\ldots,17\}$ and $z$ is
$(4,\lambda)$-admissible. Then there exists a $(z,4,\lambda)$-BIBD with a
$(\lfloor\frac{z}{2}\rfloor,\lfloor\frac{z}{2}\rfloor)$-blocking system.
\end{lemma}

\begin{proof}[{\bf Proof}]
Let $\lambda_{\min}$ be the smallest positive integer such that $z$ is
$(4,\lambda_{\min})$-admissible, and note that $\lambda \equiv 0 \mod \lambda_{\min})$. It suffices
to find a $(z,4,\lambda_{\min})$-BIBD with a
$(\lfloor\frac{z}{2}\rfloor,\lfloor\frac{z}{2}\rfloor)$-blocking system (since we can take
$\frac{\lambda}{\lambda_{\min}}$ copies of every block in this design).

For each $z \in \{6,7,\ldots,17\}$, a $(z,4,\lambda_{\min})$-BIBD with a
$(\lfloor\frac{z}{2}\rfloor,\lceil\frac{z}{2}\rceil)$-blocking system is given explicitly in
\cite{HoLiPh1} or \cite{HoLiPh2}. If $z$ is even this gives us the required result immediately, and
in each of the cases $z \in \{7,9,11,13,15,17\}$ it is routine to check that the given design in
fact admits a $(\lfloor\frac{z}{2}\rfloor,\lfloor\frac{z}{2}\rfloor)$-blocking system as required.
\end{proof}

\section{Examples of $c$-chromatic BIBDs}

In this section we will construct, for all positive integers $c$, $k$ and $\lambda$ with $c \geq
2$, $k \geq 4$ and $(c,k) \neq (2,4)$, $c$-chromatic BIBDs with block size $k$ and index $\lambda$
whose orders satisfy various congruence conditions (see Lemma \ref{CChromaticExample}). The reason
for these particular congruence conditions will become apparent when we employ these examples in
Section 5 to establish the asymptotic existence of $c$-chromatic BIBDs for each $c \geq 2$. Our
approach in this section is inspired by a technique used in \cite{DePhRo}, and also bears
similarities to methods used in \cite{HoPi}. Before proving Lemma \ref{CChromaticExample}, we
require three preliminary lemmas.

\begin{lemma}\label{CChromaticPartialBIBD}
Let $c$ and $k$ be integers such that $k \geq 3$ and $c \geq 2$. Then there exists a $c$-chromatic
partial BIBD with block size $k$ and index $1$.
\end{lemma}

\begin{proof}[{\bf Proof}] It was shown in \cite{ErHa} (and later proved constructively in \cite{Lo}) that for any
integers $k' \geq 3$ and $c' \geq 1$ there is a partial BIBD with block size $k'$ and index $1$
which has chromatic number at least $c'$. Let $(U,\mathcal{A})$ be a partial BIBD with block size
$k$ and index $1$ which has chromatic number at least $c$.

Let $\mathcal{A} =\{A_1,A_2,\ldots,A_t\}$. We will show that there is an $s \in \{1,2,\ldots,t\}$
such that the partial BIBD $(U,\{A_1,A_2,\ldots,A_s\})$ is $c$-chromatic. We claim that, for each
$i \in \{1,2,\ldots,t-1\}$, if $(U,\{A_1,A_2,\ldots,A_i\})$ has chromatic number $c^{\dag}$ then
$(U,\{A_1,A_2,\ldots,A_{i+1}\})$ has chromatic number $c^{\dag}$ or $c^{\dag}+1$. To see this,
observe that we can obtain a $(c^{\dag}+1)$-colouring of $(U,\{A_1,A_2,\ldots,A_{i+1}\})$ by
taking a $(c^{\dag})$-colouring of $(U,\{A_1,A_2,\ldots,A_i\})$ and recolouring an arbitrary
vertex of $A_{i+1}$ with a colour which is not used in the original colouring. Thus, since
$(U,\{A_1\})$ has chromatic number $2$ and $(U,\{A_1,A_2,\ldots,A_t\})$ has chromatic number at
least $c$, it follows that there is indeed an $s \in \{1,2,\ldots,t\}$ such that
$(U,\{A_1,A_2,\ldots,A_s\})$ is $c$-chromatic. \end{proof}

\begin{lemma}\label{TwoTDs}
Let $k$ be an integer such that $k \geq 5$ and let $p$ be a prime such that $p \geq k$ and if $k=5$
then $p \equiv 1 \mod 4)$. Then there exists a transversal design with group size $p$ and block
size $k$, $(V,\mathcal{G},\mathcal{B})$, such that
\begin{itemize}
    \item
$(V,\mathcal{B})$ has a $2$-blocking system such that each set of the blocking system
intersects each group in $\mathcal{G}$ in exactly $\frac{p-1}{2}$ points; and
    \item
there is a block $B^* \in \mathcal{B}$ such that $(V,\mathcal{B}\setminus\{B^*\})$ has a
$2$-blocking system such that each set of the blocking system is disjoint from $B^*$ and
intersects each group in $\mathcal{G}$ in exactly $\frac{p-1}{2}$ points.
\end{itemize}
\end{lemma}

\begin{proof}[{\bf Proof}] Let $V = \Z_k \times \Z_p$ and let $\mathcal{G} = \{\{x\} \times \Z_p:x \in \Z_k\}$. Let $I=
\{-\frac{p-1}{2},-\frac{p-3}{2},\ldots,\frac{p-3}{2},\frac{p-1}{2}\}$. For all $i \in I$ and $j \in
\Z_p$ let
$$B_{i,j}=\{(x,ix+j):x \in \Z_k\}$$
where the second coordinates are considered modulo $p$ (here, we could equally say for all $i \in
\Z_p$, but it will help later to consider $i$ as an element of $I$). Let $\mathcal{B}=\{B_{i,j}:i
\in I \hbox{ and } j \in \Z_p\}$. We claim that $(V,\mathcal{G},\mathcal{B})$ is a transversal
design.

It is easy to see that $\mathcal{B}$ contains exactly $p^2$ blocks of size $k$. Also, if a pair of
points in different groups appears in the blocks $B_{i,j}$ and $B_{i',j'}$ for some $i,i' \in I$
and $j,j' \in \Z_p$ then it is clear that $i\ell=i'\ell \mod p)$ for some $\ell \in
\{1,2,\ldots,k-1\}$. So, since $k \leq p$, it follows that $i=i'$ and hence $j=j'$. Thus,
$(V,\mathcal{G},\mathcal{B})$ is indeed a transversal design. We will complete the proof by finding
a $2$-blocking system for $(V,\mathcal{B})$ such that each set of the blocking system intersects
each group in $\mathcal{G}$ in exactly $\frac{p-1}{2}$ points, and a $2$-blocking system for
$(V,\mathcal{B} \setminus \{B_{0,0}\})$ such that each set of the blocking system is disjoint from
$B_{0,0}$ and intersects each group in $\mathcal{G}$ in exactly $\frac{p-1}{2}$ points.

Let
\begin{align*}
  S_1 =& (\{0,1,\ldots,k-4\} \times \{1,2,\ldots,\tfrac{p-1}{2}\}) \cup (\{k-3,k-1\} \times \{0,1,\ldots,\tfrac{p-3}{2}\}) \cup \\
&(\{k-2\} \times \{\tfrac{p+1}{2},\tfrac{p+3}{2},\ldots,p-1\}) \hbox{ and} \\
  S_2 =& (\{0,1,\ldots,k-4\} \times \{\tfrac{p+1}{2},\tfrac{p+3}{2},\ldots,p-1\}) \cup (\{k-3,k-1\} \times \{\tfrac{p-1}{2},\tfrac{p+1}{2},\ldots,p-2\}) \cup \\
&(\{k-2\} \times \{0,1,\ldots,\tfrac{p-3}{2}\}).
\end{align*}
We claim that $\{S_1,S_2\}$ is a $2$-blocking system for $(V,\mathcal{B})$. Suppose for a
contradiction that there exist $a \in I$ and $b \in \Z_p$ such that $B_{a,b} \cap S_2 = \emptyset$.
Then
\begin{itemize}
    \item[(1)]
$\{ax+b: x \in \{0,1,\ldots,k-4\}\} \subseteq \{0,1,\ldots,\tfrac{p-1}{2}\}$;
    \item[(2)]
$a(k-3)+b \in \{0,1,\ldots,\tfrac{p-3}{2},p-1\}$;
    \item[(3)]
$a(k-2)+b \in \{\tfrac{p-1}{2},\tfrac{p+1}{2},\ldots,p-1\}$; and
    \item[(4)]
$a(k-1)+b \in \{0,1,\ldots,\tfrac{p-3}{2},p-1\}$.
\end{itemize}
From (1) and (2) it can be seen that $(k-3)|a| \leq \frac{p+1}{2}$, unless $a=\frac{p-1}{2}$, $b=0$
and $k=5$ in which case (3) is violated. Thus, if $k \geq 6$ then $|a| \leq \frac{p+1}{6} <
\frac{p-1}{4}$ since $p \geq k$, and, if $k=5$ then $|a| \leq \frac{p-1}{4}$ since $p \equiv 1 \mod
4)$ in this case. So in all cases $a \in \{-\lfloor\frac{p-1}{4}\rfloor,
-\lfloor\frac{p-5}{4}\rfloor,\ldots, \lfloor\frac{p-1}{4}\rfloor\}$ and it follows from (1), (2)
and (3) that $a(k-1)+b \in \{\tfrac{p-1}{2},\tfrac{p+1}{2},\ldots,p-2\}$, a contradiction to (4).
It can be similarly shown that no block in $\mathcal{B}$ is disjoint from $S_1$.

Let
\begin{align*}
  T_1 &= (\{0,1,\ldots,k-3,k-1\} \times \{1,2,\ldots,\tfrac{p-1}{2}\}) \cup (\{k-2\} \times \{\tfrac{p+1}{2},\tfrac{p+3}{2},\ldots,p-1\}) \hbox{ and} \\
  T_2 &= (\{0,1,\ldots,k-3,k-1\} \times \{\tfrac{p+1}{2},\tfrac{p+3}{2},\ldots,p-1\}) \cup (\{k-2\} \times \{1,2,\ldots,\tfrac{p-1}{2}\}).
\end{align*}
We claim that $\{T_1,T_2\}$ is a $2$-blocking system for $(V,\mathcal{B} \setminus \{B_{0,0}\})$.
Suppose for a contradiction that there exist $a \in I$ and $b \in \Z_p$ such that $(a,b) \neq
(0,0)$ and $B_{a,b} \cap T_2 = \emptyset$. Then
\begin{itemize}
    \item[(1)]
$\{ax+b: x \in \{0,1,\ldots,k-3\}\} \subseteq \{0,1,\ldots,\tfrac{p-1}{2}\}$;
    \item[(2)]
$a(k-2)+b \in \{0,\tfrac{p+1}{2},\tfrac{p+3}{2},\ldots,p-1\}$; and
    \item[(3)]
$a(k-1)+b \in \{0,1,\ldots,\tfrac{p-1}{2}\}$.
\end{itemize}
From (1) it can be seen that $(k-3)|a| \leq \frac{p-1}{2}$. Thus, since $k \geq 5$, $|a| \leq
\frac{p-1}{4}$. So $a \in \{-\lfloor\frac{p-1}{4}\rfloor, -\lfloor\frac{p-5}{4}\rfloor,\ldots,
\lfloor\frac{p-1}{4}\rfloor\}$ and it follows from (1) and (2) that $a(k-1)+b \in
\{\tfrac{p+1}{2},\tfrac{p+3}{2},\ldots,p-1\}$, a contradiction to (3). It can be similarly shown
that no block in $\mathcal{B} \setminus \{B_{0,0}\}$ is disjoint from $T_1$. \end{proof}

\begin{lemma}\label{TwoTDsBlockSize4}
There exists a transversal design with group size $13$ and block size $4$,
$(V,\mathcal{G},\mathcal{B})$, such that
\begin{itemize}
    \item
$(V,\mathcal{B})$ has a $3$-blocking system, each set of which intersects each group in
$\mathcal{G}$ in exactly $4$ points; and
    \item
there exists a block $B^* \in \mathcal{B}$ such that $(V,\mathcal{B}\setminus\{B^*\})$ has a
$3$-blocking system each set of which is disjoint from $B^*$ and intersects each group in
$\mathcal{G}$ in exactly $4$ points.
\end{itemize}
Also, for each positive integer $\lambda$, there exists a $(13,4,\lambda)$-BIBD with a
$(4,4,4)$-blocking system.
\end{lemma}

\begin{proof}[{\bf Proof}] Let $V = \Z_4 \times \Z_{13}$ and $\mathcal{G} = \{\{x\} \times \Z_{13}: x \in \Z_4\}$. For
all $i,j \in \Z_{13}$ let
$$B_{i,j}=\{(0,i),(1,j),(2,i+j),(3,i+2j)\},$$
and let $\mathcal{B} = \{B_{i,j}:i,j \in \Z_{13}\}$. Then $(V,\mathcal{G},\mathcal{B})$ is a
transversal design with group size $13$ and block size $4$. Let {\small
\begin{align*}
  S_1 &= \{(0,1),(0,2),(0,3),(0,4),(1,1),(1,2),(1,3),(1,4),(2,1),(2,2),(2,3),(2,4),(3,1),(3,2),(3,8),(3,9)\}; \\
  S_2 &= \{(0,5),(0,6),(0,7),(0,8),(1,5),(1,6),(1,7),(1,10),(2,7),(2,8),(2,9),(2,12),(3,0),(3,3),(3,10),(3,11)\}; \\
  S_3 &= \{(0,0),(0,9),(0,10),(0,11),(1,0),(1,8),(1,11),(1,12),(2,0),(2,5),(2,6),(2,10),(3,4),(3,5),(3,6),(3,7)\}.
\end{align*}}
Then $\{S_1,S_2,S_3\}$ is a blocking system for $(V,\mathcal{B})$, each set of which intersects
each group in $\mathcal{G}$ in exactly $4$ points. Let {\small
\begin{align*}
  T_1 &= \{(0,2),(0,3),(0,4),(0,5),(1,2),(1,3),(1,4),(1,5),(2,1),(2,3),(2,4),(2,5),(3,1),(3,2),(3,10),(3,11)\}; \\
  T_2 &= \{(0,6),(0,7),(0,8),(0,9),(1,6),(1,7),(1,9),(1,10),(2,8),(2,9),(2,10),(2,11),(3,0),(3,4),(3,5),(3,12)\}; \\
  T_3 &= \{(0,0),(0,10),(0,11),(0,12),(1,0),(1,8),(1,11),(1,12),(2,0),(2,6),(2,7),(2,12),(3,6),(3,7),(3,8),(3,9)\}.
\end{align*}}
Then $\{T_1,T_2,T_3\}$ is a blocking system for $(V,\mathcal{B}\setminus\{B_{1,1}\})$ each set of
which is disjoint from $B_{1,1}$ and intersects each group in $\mathcal{G}$ in exactly $4$ points.

We saw in the proof of Lemma \ref{2ChromaticBlockSize4} that there exists a $(13,4,1)$-BIBD, and
it is easy to show that any such design must have a $(4,4,4)$-blocking system. By taking $\lambda$
copies of each block in this design, we can obtain the required $(13,4,\lambda)$-BIBD.
\end{proof}

The blocking systems $\{S_1,S_2,S_3\}$ and $\{T_1,T_2,T_3\}$ in the above proof were found by
computer search.

\begin{lemma}\label{CChromaticExample}
Let $c$, $k$, $\lambda$ and $m$ be positive integers such that $c \geq 2$, $k \geq 4$ and $m \equiv
0 \mod k(k-1))$. Further suppose that if $k=4$, then $c \geq 3$ and $m \not\equiv 0 \mod 13)$. For
each $(k,\lambda)$-admissible integer $\ell \in \{0,1,\ldots,m-1\}$ there is an integer $w$ such
that $w > m$, $w \equiv \ell \mod m)$ and there exists a $c$-chromatic $(w,k,\lambda)$-BIBD with an
$(s_1,s_2,\ldots,s_c)$-blocking system for some integers $s_1,s_2,\ldots,s_c$ satisfying $s_i \leq
\lfloor\frac{w-1}{2}\rfloor$ for each $i \in \{1,2,\ldots,c\}$.
\end{lemma}

\begin{proof}[{\bf Proof}] Let $\ell$ be a $(k,\lambda)$-admissible element of $\{0,1,\ldots,m-1\}$. We will first deal
with the case $k \geq 5$. The special case $k=4$ will be dealt with later.

By Dirichlet's Theorem there are infinitely many primes congruent to $1$ modulo $k(k-1)$. Thus, by
Lemma \ref{2Chromatic1ModThing}, we can choose $p$ to be an odd prime such that $p \equiv 1 \mod
k(k-1))$, $\gcd(p,m)=1$ and there exists a $(p,k,\lambda)$-BIBD with a
$(\frac{p-1}{2},\frac{p-1}{2})$-blocking system (note that any integer congruent to $1$ modulo
$k(k-1)$ is $(k,\lambda)$-admissible).

By Lemma \ref{CChromaticPartialBIBD} there is a $c$-chromatic partial BIBD with block size $k$ and
index $1$. Clearly, by adding points and blocks to this design in such a way that each new block
is disjoint from each other block of the design, we can produce, for some positive integer $u$, a
$c$-chromatic partial $(u,k,1)$-BIBD $(U,\mathcal{A}_1)$ such that $\gcd(|\mathcal{A}_1|,m)=1$ and
there is a point in $U$ which is in exactly one block in $\mathcal{A}_1$. Let $b=|\mathcal{A}_1|$
and let $\{R_1,R_2,\ldots,R_c\}$ be a blocking system for $(U,\mathcal{A}_1)$.

Let $G$ be the graph on vertex set $U$ in which two vertices are adjacent if and only if the
corresponding pair of points is contained in a block in $\mathcal{A}_1$. We claim that Theorem
\ref{LamWilTheorem} implies there is a decomposition of the $\lambda$-fold complete graph of order
$v$ into copies of $G$ for all sufficiently large integers $v$ such that $\lambda(v-1) \equiv 0
\mod k-1)$ and $\lambda v(v-1) \equiv 0 \mod bk(k-1))$. To see this we take $C$ in Theorem
\ref{LamWilTheorem} as a set containing a single colour, let $H$ be $G$ considered as a symmetric
digraph of this colour, and observe that we have $\mu(H)=2|E(G)|=bk(k-1)$ and, for each $x \in U$,
$\tau(H,x)=(\deg_G(x),\deg_G(x))$. So $\beta(\{H\})=bk(k-1)$ and $\alpha(\{H\}) =
\gcd(\{\deg_G(x):x \in V(G)\})=k-1$ (note that, for each $x \in U$, $\deg_G(x)=r_x(k-1)$ where
$r_x$ is the number of blocks in $\mathcal{A}_1$ which contain $x$, and that we have seen that
$r_y=1$ for some $y \in U$). Thus our claim does indeed follow from Theorem \ref{LamWilTheorem}.

Since $\gcd(p,m)=1$ and $\gcd(b,m)=1$, by the Chinese Remainder Theorem there are infinitely many
positive integers which are congruent to $1$ modulo $b$ and whose product with $p$ is congruent to
$\ell$ modulo $m$. Thus, there is an integer $y$ such that $y \geq 2u$, $py > m$, $py \equiv \ell
\mod m)$, $y \equiv 1 \mod b)$ and there is a decomposition of the $\lambda$-fold complete graph
of order $y$ into copies of $G$ (note that since $\ell$ is $(k,\lambda)$-admissible, since $p
\equiv 1 \mod k(k-1))$, since $m \equiv 0 \mod k(k-1))$, and since $\gcd(b,m)=1$, the congruences
imply that $\lambda (y-1) \equiv 0 \mod k-1)$ and $\lambda y(y-1) \equiv 0 \mod bk(k-1))$).
Clearly then, there is an embedding of $(U,\mathcal{A}_1)$ in a $(y,k,\lambda)$-BIBD
$(Y,\mathcal{A}_1 \cup \mathcal{A}_2)$.

Let $Z$ be a set such that $|Z|=p$. Let $z^* \in Z$ and let $\{Z_1,Z_2\}$ be a partition of $Z
\setminus \{z^*\}$ such that $|Z_1|=|Z_2|=\frac{p-1}{2}$. Let $V=Y \times Z$ be a point set. Let
$S_1=(Y \times Z_1) \cup (R_1 \times \{z^*\})$, $S_2=(Y \times Z_2) \cup (R_2 \times \{z^*\})$, and
$S_i=R_i \times \{z^*\}$ for each $i \in \{3,4,\ldots,c\}$. Note that $S_1,S_2,\ldots,S_c$ are
pairwise disjoint and that, since $y \geq 2u$, $|S_i| \leq \lfloor\frac{py-1}{2}\rfloor$ for each
$i \in \{1,2,\ldots,c\}$. We will construct a collection of blocks $\mathcal{C}$ such that
$(V,\mathcal{C})$ is a $(py,k,\lambda)$-BIBD for which $\{S_1,S_2,\ldots,S_c\}$ is a blocking
system, and such that $\mathcal{C}$ contains an isomorphic copy of $\mathcal{A}_1$. This will
complete the proof since $py >m$, since $py \equiv \ell \mod m)$ and since the fact that
$\mathcal{C}$ contains an isomorphic copy of $\mathcal{A}_1$ implies that $(V,\mathcal{C})$ has
chromatic number at least $c$.

For each $A \in \mathcal{A}_1$, let $\mathcal{B}_A$ be a collection of blocks such that $(A\times
Z,\{\{x\} \times Z:x \in A\},\mathcal{B}_A)$ is a transversal design with group size $p$ and block
size $k$ such that $A \times \{z^*\} \in \mathcal{B}_A$ and $\{A \times Z_1,A \times Z_2\}$ is a
blocking system for $(A\times Z,\mathcal{B}_A\setminus \{A \times \{z^*\}\})$ (such a collection
exists by Lemma \ref{TwoTDs}, noting that if $k=5$ then $p \equiv 1 \mod 4)$ since $p \equiv 1 \mod
k(k-1))$). For each $A \in \mathcal{A}_2$, let $\mathcal{B}^{\dag}_A$ be a collection of blocks
such that $(A\times Z,\{\{x\} \times Z:x \in A\},\mathcal{B}^{\dag}_A)$ is a transversal design
with group size $p$ and block size $k$ for which $\{A \times Z_1,A \times Z_2\}$ is a blocking
system (such a collection exists by Lemma \ref{TwoTDs}, noting that if $k=5$ then $p \equiv 1 \mod
4)$ since $p \equiv 1 \mod k(k-1))$). For each $x \in Y$, let $\mathcal{B}_x$ be a collection of
blocks such that $(\{x\}\times Z,\mathcal{B}_x)$ is a $(p,k,\lambda)$-BIBD for which $\{\{x\}
\times Z_1,\{x\} \times Z_2\}$ is a blocking system (such a collection exists by the definition of
$p$).

Let
$$\mathcal{C} = \left(\bigcup_{A \in \mathcal{A}_1} \mathcal{B}_A \right) \cup \left(\bigcup_{A \in \mathcal{A}_2} \mathcal{B}^{\dag}_A \right) \cup \left(\bigcup_{x \in Y} \mathcal{B}_x \right).$$
We claim that $(V,\mathcal{C})$ is a $(py,k,\lambda)$-BIBD for which $\{S_1,S_2,\ldots,S_c\}$ is a
blocking system and that $\mathcal{C}$ contains an isomorphic copy of $(U,\mathcal{A}_1)$ (on the
point set $U \times \{z^*\}$), which will suffice to complete the proof in the case $k=5$. Here,
we will verify this claim in some detail, but later in the paper we will leave similar
verifications to the reader.

Routine case analysis shows that each pair of points in $V$ is in exactly $\lambda$ blocks in
$\mathcal{C}$ and hence that $(V,\mathcal{C})$ is a $(py,k,\lambda)$-BIBD. For each $A \in
\mathcal{A}_1$, we have that $A \times \{z^*\} \in \mathcal{B}_A$, and it follows that
$\mathcal{C}$ contains an isomorphic copy of $(U,\mathcal{A}_1)$ on the point set $U \times
\{z^*\}$. Furthermore, because $\{R_1,R_2,\ldots,R_c\}$ is a blocking system for
$(U,\mathcal{A}_1)$, each block in this copy of $(U,\mathcal{A}_1)$ intersects at least two sets
in $\{R_i \times \{z^*\}:i\in\{1,2,\ldots,c\}\}$ and hence at least two sets in
$\{S_1,S_2,\ldots,S_c\}$ (note $R_i \times \{z^*\} \subseteq S_i$ for each
$i\in\{1,2,\ldots,c\}$). Finally, because of the blocking systems possessed by the designs in
$\{(A\times Z,\mathcal{B}_A\setminus \{A \times \{z^*\}\}):A \in \mathcal{A}_1\}$, $\{(A\times
Z,\{\{x\} \times Z:x \in A\},\mathcal{B}^{\dag}_A):A \in \mathcal{A}_2\}$, and $\{(\{x\}\times
Z,\mathcal{B}_x):x \in Y\}$, every block in $\mathcal{C}$ which is not in the copy of
$(U,\mathcal{A}_1)$ intersects both $Y \times Z_1$ and $Y \times Z_2$ and hence both $S_1$ and
$S_2$ (note that $Y \times Z_1 \subseteq S_1$ and  $Y \times Z_2 \subseteq S_2$). So
$(V,\mathcal{C})$ is indeed a $(py,k,\lambda)$-BIBD for which $\{S_1,S_2,\ldots,S_c\}$ is a
blocking system and $\mathcal{C}$ does contain an isomorphic copy of $(U,\mathcal{A}_1)$, as
required.

In the case $k=4$ note that $c \geq 3$ and choose $p=13$. Note that $p \equiv 1 \mod k(k-1))$, that
$\gcd(p,m)=1$ since $m \not\equiv 0 \mod 13)$, and that there is a $(p,k,\lambda)$-BIBD with a
$(4,4,4)$-blocking system by Lemma \ref{TwoTDsBlockSize4}. Also by Lemma \ref{TwoTDsBlockSize4},
there exists a transversal design with group size $13$ and block size $4$,
$(V,\mathcal{G},\mathcal{B})$, such that
\begin{itemize}
    \item
$(V,\mathcal{B})$ has a $3$-blocking system, each set of which intersects each group in
$\mathcal{G}$ in exactly $4$ points; and
    \item
there exists a block $B^* \in \mathcal{B}$ such that $(V,\mathcal{B}\setminus\{B^*\})$ has a
$3$-blocking system each set of which is disjoint from $B^*$ and intersects each group in
$\mathcal{G}$ in exactly $4$ points.
\end{itemize}
By using a similar argument to that used in the case $k \geq 5$ we can obtain the required block
design.
\end{proof}

\section{Asymptotic existence of $c$-chromatic BIBDs}

We are now almost ready to prove Theorem \ref{MainTheorem} in the case $k \geq 4$. The final
preliminary result we require uses Wilson's fundamental construction to obtain GDDs with a large
number of groups of large size which possess $2$-blocking systems with certain properties.

\begin{lemma}\label{TwoDifferentGDD}
Let $k$ and $\lambda$ be positive integers such that $k \geq 4$. Then there exist positive integers
$t$ and $a_0$ such that if $a$, $a^{\dag}$ and $a^{\ddag}$ are integers such that $a \geq a_0$,
$a^{\dag} \leq a$, $a^{\ddag} \leq a$ and $a \equiv a^{\dag} \equiv a^{\ddag} \equiv 0 \mod
k(k-1))$, then there exists a $(k,\lambda)$-GDD of type $a^t(a^{\dag})^1(a^{\ddag})^1$ which has a
$2$-blocking system such that each set of the blocking system intersects each group $G$ of the GDD
in exactly $\frac{|G|}{2}$ points.
\end{lemma}

\begin{proof}[{\bf Proof}]
By Lemma \ref{ManyGroupGDD} or Lemma \ref{ManyGroupGDDk=4}, it is easy to see that there is a
positive integer $t$ such that for each $s \in \{t,t+1,t+2\}$ there exists a $(k,\lambda)$-GDD of
type $(k(k-1))^s$ with a $2$-blocking system such that each set of the blocking system intersects
each group of the design in exactly $\frac{k(k-1)}{2}$ points. The main result of \cite{ChErSt}
implies that, for a given positive integer $k'$, there exists a transversal design with group size
$g'$ and block size $k'$ for all sufficiently large integers $g'$. Thus, there is an integer $g_0$
such that for any integer $g \geq g_0$ there exists a transversal design with group size $g$ and
block size $t+2$. Let $a$, $a^{\dag}$ and $a^{\ddag}$ be integers such that $a \geq g_0k(k-1)$,
$a^{\dag} \leq a$, $a^{\ddag} \leq a$ and $a \equiv a^{\dag} \equiv a^{\ddag} \equiv 0 \mod
k(k-1))$. Then there exists a transversal design with group size $\frac{a}{k(k-1)}$ and block size
$t+2$. By deleting some points from this transversal design we can obtain a
$(\{t,t+1,t+2\},1)$-GDD $(V,\mathcal{F},\mathcal{A})$ of type $(\frac{a}{k(k-1)})^t
(\frac{a^{\dag}}{k(k-1)})^1 (\frac{a^{\ddag}}{k(k-1)})^1$.

Let $Z$ be a set with $|Z|=k(k-1)$ and let $\{Z_1,Z_2\}$ be a partition of $Z$ with
$|Z_1|=|Z_2|=\frac{k(k-1)}{2}$. Let $\mathcal{G} = \{F \times Z : F \in \mathcal{F}\}$. For each
block $A \in \mathcal{A}$, let $\mathcal{B}_A$ be a collection of blocks  such that $(A \times Z,
\{\{x\} \times Z:x \in A\}, \mathcal{B}_A)$ is a $(k,\lambda)$-GDD of type $(k(k-1))^{|A|}$ for
which $\{A \times Z_1,A \times Z_2\}$ is a blocking system (such a collection exists since $|A| \in
\{t,t+1,t+2\}$ and $|Z_1|=|Z_2|=\frac{k(k-1)}{2}$). Let
$$\mathcal{B} = \bigcup_{A \in \mathcal{A}}\mathcal{B}_A.$$
It is routine to check that $(V \times Z, \mathcal{G},\mathcal{B})$ is a $(k,\lambda)$-GDD of type
$a^t(a^{\dag})^1(a^{\ddag})^1$ for which $\{V \times Z_1,V \times Z_2\}$ is a blocking system.
Since $|(V \times Z_1) \cap G| = |(V \times Z_2) \cap G| = \frac{|G|}{2}$ for each $G \in
\mathcal{G}$, the proof is complete.
\end{proof}

\begin{proof}[{\bf Proof of Theorem \ref{MainTheorem} in the case $\boldsymbol{k \geq 4}$.}]
It is known that for any positive integer $\lambda$, a $2$-chromatic $(v,4,\lambda)$-BIBD exists
for each $(4,\lambda)$-admissible integer $v$ (see \cite{FrGrLiRo,HoLiPh1,HoLiPh2,RoCo}), so we
may assume that if $k=4$ then $c \geq 3$. Since there are only finitely many congruence classes
modulo $k(k-1)$ it suffices to show that, for each $(k,\lambda)$-admissible integer $\ell' \in
\{0,1,\ldots,k(k-1)-1\}$, there is a $c$-chromatic $(v,k,\lambda)$-BIBD for each sufficiently
large integer $v$ such that $v \equiv \ell' \mod k(k-1))$.

Let $\ell' \in \{0,1,\ldots,k(k-1)-1\}$ be a $(k,\lambda)$-admissible integer. We will first deal
with the case where $k \geq 5$ or where $k=4$ and $\ell'$ is odd. The special case where $k=4$ and
$\ell'$ is even will be dealt with later. By Lemma \ref{2ChromaticInEachCongClass} or Lemma
\ref{2ChromaticBlockSize4} there is a positive integer $u$ such that $u \equiv \ell' \mod k(k-1))$,
if $k \geq 5$ then $u \geq 2k-1$, if $k=4$ then $u \in \{7,9,11,13,15,17\}$, and there exists a
$(u,k,\lambda)$-BIBD with a $(\lfloor\frac{u-1}{2}\rfloor,\lfloor\frac{u-1}{2}\rfloor)$-blocking
system. Let $m=\lcm(k(k-1),u-1)$. Now, since there are only finitely many congruence classes modulo
$m$ it suffices to show that, for each $(k,\lambda)$-admissible integer $\ell'' \in
\{0,1,\ldots,m-1\}$ for which $\ell'' \equiv \ell' \mod k(k-1))$, there is a $c$-chromatic
$(v,k,\lambda)$-BIBD for each sufficiently large integer $v$ such that $v \equiv \ell'' \mod m)$.

Let $\ell'' \in \{0,1,\ldots,m-1\}$ be a $(k,\lambda)$-admissible integer for which $\ell'' \equiv
\ell' \mod k(k-1))$. Note the following facts.
\begin{itemize}
    \item[(i)]
By Lemma \ref{TwoDifferentGDD} there are positive integers $t$ and $a_0$ such that for any
integers $x$, $x^{\dag}$ and $x^{\ddag}$ such that $x \geq a_0$, $x^{\dag} \leq x$, $x^{\ddag}
\leq x$ and $x \equiv x^{\dag} \equiv x^{\ddag} \equiv 0 \mod k(k-1))$ there exists a
$(k,\lambda)$-GDD of type $x^t(x^{\dag})^1(x^{\ddag})^1$ which has a $2$-blocking system such
that each set of the blocking system intersects each group $G$ of the GDD in exactly
$\frac{|G|}{2}$ points.
    \item[(ii)]
Since there exists a $(u,k,\lambda)$-BIBD with a
$(\lfloor\frac{u-1}{2}\rfloor,\lfloor\frac{u-1}{2}\rfloor)$-blocking system, by Lemma
\ref{BasicConstructionsInfinity} (b) there is a positive integer $n_0$ such that, for each
integer $n \geq n_0$ with $n \equiv 0 \mod m)$, there exists a $(k,\lambda)$-GDD of type
$u^11^n$ which has an
$(\frac{n}{2}+\lfloor\frac{u}{2}\rfloor,\frac{n}{2}+\lfloor\frac{u}{2}\rfloor)$-blocking system
such that each set of the blocking system intersects the group of size $u$ in exactly
$\lfloor\frac{u}{2}\rfloor$ points (note that $n+u$ is $(k,\lambda)$-admissible since $m \equiv
0 \mod k(k-1))$, that $\lfloor\frac{u-1}{2}\rfloor \leq \lfloor\frac{u}{2}\rfloor$, and that
$\frac{n}{u-1}\lfloor\frac{u-1}{2}\rfloor \leq \frac{n}{2}$).
    \item[(iii)]
By Lemma \ref{CChromaticExample} there is an integer $w > m$ such that $w \equiv \ell'' \mod
m)$ and there exists a $c$-chromatic $(w,k,\lambda)$-BIBD with an
$(s_1,s_2,\ldots,s_c)$-blocking system for some integers $s_1,s_2,\ldots,s_c$ satisfying $s_i
\leq \lfloor\frac{w-1}{2}\rfloor$ for each $i \in \{1,2,\ldots,c\}$ (if $k=4$ then $u \in
\{7,9,11,13,15,17\}$ which implies that $m \not\equiv 0 \mod 13)$).
\end{itemize}

Let $N$ be the smallest integer such that $N \equiv 0 \mod m)$ and $N \geq
\max\{a_0,n_0+m(t-1),w-u\}$. Let $v \geq Nt+n_0+m(t-1)+w$ be an integer such that $v \equiv \ell''
\mod m)$. We will construct a $c$-chromatic $(v,k,\lambda)$-BIBD to complete the proof.

It can be seen that there are unique integers $a$ and $a^{\dag}$ such that $a \equiv a^{\dag}
\equiv 0 \mod m)$, $v-w=at+a^{\dag}$ and $n_0 \leq a^{\dag} \leq n_0+m(t-1)$ (note that $v \equiv w
\mod m)$). Now since $v \geq Nt+n_0+m(t-1)+w$, we have that $at+a^{\dag} \geq Nt+n_0+m(t-1)$ and
thus, since $a^{\dag} \leq n_0+m(t-1)$, we have that $a \geq N$. Note that $N \geq a_0$, that $N
\geq n_0+m(t-1) \geq a^{\dag}$, that $N \geq w-u$ and, since $m \equiv 0 \mod k(k-1))$, that $a
\equiv a^{\dag} \equiv w-u \equiv 0 \mod k(k-1))$. Thus, by (i) there exists a $(k,\lambda)$-GDD
$(V, \mathcal{F} \cup \{F^{\dag}, F^{\ddag}\}, \mathcal{A})$ of type $a^t(a^{\dag})^1(w-u)^1$,
where $|F|=a$ for all $F \in \mathcal{F}$, $|F^{\dag}|=a^{\dag}$ and $|F^{\ddag}| = w-u$, which has
a $2$-blocking system $\{R_1,R_2\}$ such that $|R_1 \cap F| = |R_2 \cap F|=\frac{|F|}{2}$ for each
group $F \in \mathcal{F} \cup \{F^{\dag}, F^{\ddag}\}$.

Let $U$ be a set disjoint from $V$ such that $|U|=u$ and let $U_1$ and $U_2$ be disjoint subsets of
$U$ with $|U_1|=|U_2|=\lfloor\frac{u}{2}\rfloor$. We will now construct a $c$-chromatic
$(v,k,\lambda)$-BIBD on the point set $V \cup U$.
\begin{itemize}
    \item
For each group $F \in \mathcal{F}$, let $\mathcal{B}_F$ be a collection of blocks such that $(F
\cup U,\{U\} \cup \{\{f\}:f \in F\},\mathcal{B}_F)$ is a $(k,\lambda)$-GDD of type $u^11^a$ for
which $\{(F \cap R_1) \cup U_1,(F \cap R_2) \cup U_2\}$ is a blocking system. Such collections
exist by (ii) since $a \geq n_0$, $a \equiv 0 \mod m)$, $|F \cap R_1| = |F \cap R_2| =
\frac{a}{2}$ for all $F \in \mathcal{F}$, and $|U_1|=|U_2|=\lfloor\frac{u}{2}\rfloor$.
    \item
Let $\mathcal{B}^{\dag}$ be a collection of blocks such that $(F^{\dag} \cup U,\{U\} \cup
\{\{f\}:f \in F^{\dag}\},\mathcal{B}^{\dag})$ is a $(k,\lambda)$-GDD of type $u^11^{a^{\dag}}$
for which $\{(F^{\dag} \cap R_1) \cup U_1,(F^{\dag} \cap R_2) \cup U_2\}$ is a blocking system.
Such a collection exists by (ii) since $a^{\dag} \geq n_0$, $a^{\dag} \equiv 0 \mod m)$,
$|F^{\dag} \cap R_1| = |F^{\dag} \cap R_2| = \frac{a^{\dag}}{2}$, and
$|U_1|=|U_2|=\lfloor\frac{u}{2}\rfloor$.
    \item
Let $\mathcal{B}^{\ddag}$ be a collection of blocks such that $(F^{\ddag} \cup
U,\mathcal{B}^{\ddag})$ is a $c$-chromatic $(w,k,\lambda)$-BIBD which has a blocking system
$\{R^{\ddag}_1,R^{\ddag}_2,\ldots,R^{\ddag}_c\}$ such that $R^{\ddag}_1 \subseteq (F^{\ddag}
\cap R_1) \cup U_1$, $R^{\ddag}_2 \subseteq (F^{\ddag} \cap R_2) \cup U_2$, and
$\{R^{\ddag}_1,R^{\ddag}_2,\ldots,R^{\ddag}_c\}$ is a partition of $F^{\ddag} \cup U$. Such a
collection exists by (iii) since $|(F^{\ddag} \cap R_1) \cup U_1|=|(F^{\ddag} \cap R_2) \cup
U_2|=\lfloor\frac{w}{2}\rfloor$.
\end{itemize}

Let $S_1=(R_1 \setminus F^{\ddag}) \cup R^{\ddag}_1$, $S_2=(R_2 \setminus F^{\ddag}) \cup
R^{\ddag}_2$, and $S_i=R^{\ddag}_i$ for each $i \in \{3,4,\ldots,c\}$. Note that
$\{S_1,S_2,\ldots,S_c\}$ is a partition of $V \cup U$. Let
$$\mathcal{B} = \mathcal{A} \cup \left(\bigcup_{F \in \mathcal{F}}\mathcal{B}_F\right) \cup \mathcal{B}^{\dag} \cup \mathcal{B}^{\ddag}.$$
It is routine to check that $(V \cup U, \mathcal{B})$ is a $c$-chromatic $(v,k,\lambda)$-BIBD for
which $\{S_1,S_2,\ldots,S_c\}$ is a blocking system (note that no block in $\mathcal{B} \setminus
\mathcal{B}^{\ddag}$ has more than one point in $F^{\ddag} \cup U$ and hence no such block can be
a subset of any set in $\{S_3,S_4,\ldots,S_c\}$).

We now consider the special case where $k=4$ and $\ell'$ is even. We proceed exactly as we did in
the main case, with four exceptions. Firstly we note that by Lemma \ref{2ChromaticBlockSize4} there
is an integer $u$ such that $u \equiv \ell' \mod k(k-1))$, $u \in \{6,8,10,12,14,16\}$, and there
exists a $(u,k,\lambda)$-BIBD with a $(\frac{u}{2},\frac{u}{2})$-blocking system. Secondly, we
define $m=\lcm(12,u)$. Thirdly, instead of (ii) we instead observe the following.
\begin{itemize}
    \item[(ii)$'$]
Since there exists a $(u,k,\lambda)$-BIBD with a $(\frac{u}{2},\frac{u}{2})$-blocking system,
by Lemma \ref{BasicConstructionsNoInfinity} (b) there is a positive integer $n_0$ such that,
for all integers $n$ for which $n \geq n_0$ and $n \equiv 0 \mod m)$, there exists a
$(k,\lambda)$-GDD of type $u^11^n$ which has a $(\frac{u+n}{2},\frac{u+n}{2})$-blocking system
such that each set of the blocking system intersects the group of size $u$ in exactly
$\frac{u}{2}$ points (note that $n+u$ is $(k,\lambda)$-admissible since $m \equiv 0 \mod
k(k-1))$).
\end{itemize}
Lastly, in our justification of (iii) we must note that $m \not\equiv 0 \mod 13)$ since $u \in
\{6,8,10,12,14,16\}$.

Except as noted, the arguments given in the main case hold without any alteration. \end{proof}

\section{The case of block size 3}

It only remains for us to prove Theorem \ref{MainTheorem} in the case $k=3$. When $\lambda=1$ this
has already been achieved by de Brandes, Phelps and R\"{o}dl \cite{DePhRo}. In this final section
we generalise their result to cover all values of $\lambda$. The methods we employ are similar, but
not identical, to theirs.

\begin{lemma}\label{TwoTDsBlockSize3}
There exist two transversal designs $(V,\mathcal{G},\mathcal{B}^{\dag})$ and
$(V,\mathcal{G},\mathcal{B})$ with group size $3$ and block size $3$ having the same point set and
the same group set such that
\begin{itemize}
    \item
there are two distinct points $x,y \in V$ such that every block which is in $\mathcal{B}$ but
not in $\mathcal{B}^{\dag}$ contains either $x$ or $y$; and
    \item
there is a partition $\{S_1,S_2,S_3\}$ of $V$ such that
\begin{itemize}
    \item[(i)]
$|S_i \cap G|=1$ for all $i \in \{1,2,3\}$ and $G \in \mathcal{G}$;
    \item[(ii)]
$\{S_1,S_2,S_3\}$ is a blocking set for $(V,\mathcal{B}^{\dag})$; and
    \item[(iii)]
$S_1 \in \mathcal{B}$.
\end{itemize}
\end{itemize}
\end{lemma}

\begin{proof}[{\bf Proof}] Let $V = \Z_3 \times \Z_3$ and let $\mathcal{G} = \{\{x\} \times \Z_3:x \in \Z_3\}$. Let
$\rho$ be the permutation $(0\ 1)$ of $\Z_3$. Let
$\mathcal{B}^{\dag}=\{\{(0,i),(1,i+j),(2,i+2j+1)\}:i,j \in \Z_3\}$ and let
$\mathcal{B}=\{\{(0,i),(1,i+j),(2,\rho(i+2j+1))\}:i,j \in \Z_3\}$, where the addition is
considered modulo $3$. It is easy to check that $\mathcal{B}^{\dag}$ and $\mathcal{B}$ satisfy the
required conditions (take $\{x,y\}=\{(2,0),(2,1)\}$ and $S_i = \Z_3 \times \{i-1\}$ for each $i
\in \{1,2,3\}$). \end{proof}

\begin{lemma}\label{BIBDBlockSize3}
Let $w$ and $\lambda$ be positive integers such that $w \geq 5$ and $w$ is
$(3,\lambda)$-admissible. Then there exists a $(w,3,\lambda)$-BIBD with a $3$-blocking system such
that the sets of the system partition the point set of the BIBD and the sizes of any two sets of
the system differ by at most $1$.
\end{lemma}

\begin{proof}[{\bf Proof}] In Theorem 18.4 of \cite{CoRo}, it is proved that, for each positive integer $\lambda$ and
each $(3,\lambda)$-admissible integer $w$, there exists a $3$-colourable $(w,3,\lambda)$-BIBD. In
the proof, for each positive integer $\lambda$ and each $(3,\lambda)$-admissible integer $w$ such
that $w \geq 5$ and $w \notin \{6,8\}$, a $(w,3,\lambda)$-BIBD is explicitly constructed, on a
point set explicitly given as $V \times \{0,1,2\}$, $V \times \{0,1,2\} \cup \{(\infty,1)\}$ or $V
\times \{0,1,2\} \cup \{(\infty,1),(\infty,2)\}$ for some set $V$, in such a way that each block
of the BIBD contains two points with different second coordinates. Thus, the partition of the
point set suggested by the second coordinates gives a suitable blocking system. For each positive
integer $\lambda$ such that $6$ is $(3,\lambda)$-admissible, any partition of the point set of a
$(6,3,\lambda)$-BIBD into parts of size $2$ will form a suitable blocking system for the BIBD.
Finally, for each positive integer $\lambda$ such that $8$ is $(3,\lambda)$-admissible, the proof
exhibits an $(8,3,\lambda)$-BIBD whose point set contains two disjoint subsets of size $3$ which
are not blocks of the BIBD. Clearly, these two sets along with a third containing the remaining
points form a suitable blocking system for this BIBD.
\end{proof}

\begin{lemma}\label{HoleyBIBDBlockSize3}
Let $h \in \{0,1,2,3,4,5\}$ and let $\lambda$ be a positive integer such that $\lambda$ is even if
$h$ is even. Then there exists a $(3,\lambda)$-GDD $(V,\mathcal{G},\mathcal{B})$ of type $h^11^6$
with a $3$-blocking system such that the sets of the system partition the point set of the GDD,
each set of the system contains at least two points which are in groups of size $1$, and the sizes
of any two sets of the system differ by at most $1$.
\end{lemma}

\begin{proof}[{\bf Proof}] The result follows directly from Lemma \ref{BIBDBlockSize3} if $h \in \{0,1\}$, so assume
that $h \in \{2,3,4,5\}$. Let $\lambda_{\min}=1$ if $h$ is odd and $\lambda_{\min}=2$ if $h$ is
even. It suffices to find a $(3,\lambda_{\min})$-GDD of type $h^11^6$ with a $3$-blocking system
such that the sets of the system partition the point set of the GDD, each set of the system
contains at least two points which are in groups of size $1$, and the sizes of any two sets of the
system differ by at most $1$ (since we can take $\frac{\lambda}{\lambda_{\min}}$ copies of every
block in this design). Below, we give the blocks of such designs along with the sets of the
required blocking systems. In the interests of space we give the blocks in columns. In each case
the point set of the design is taken to be $\{0,1,\ldots,h-1\} \cup \{a,b,c,d,e,f\}$ where
$\{0,1,\ldots,h-1\}$ is the group of size $h$. The existence of such designs (not considering
blocking systems) was first established in \cite{St}.

\begin{center}
\begin{tabular}{|l|c|l|l|}
  \hline
  $h$&$\lambda_{\min}$& blocks & blocking system sets  \\  \hline
     &   & $\mathtt{000000111111aaabbc}$ & $\mathtt{\{0,b,d\}}$  \\[-1.8ex]
  $2$& 2 & $\mathtt{aabbeeaabcdebecddd}$ & $\mathtt{\{1,a,c\}}$  \\[-1.8ex]
     &   & $\mathtt{ddeeffbfcdefcfeffe}$ & $\mathtt{\{e,f\}}$  \\ \hline
     &   & $\mathtt{000111222ab}$ & $\mathtt{\{0,b,d\}}$  \\[-1.8ex]
  $3$& 1 & $\mathtt{abcaceabdcd}$ & $\mathtt{\{1,a,c\}}$  \\[-1.8ex]
     &   & $\mathtt{defbdffceef}$ & $\mathtt{\{2,e,f\}}$  \\ \hline
     &   & $\mathtt{000000111111222222333333ab}$ & $\mathtt{\{0,3,b,d\}}$  \\[-1.8ex]
  $4$& 2 & $\mathtt{aabbcdaabbccaabbddaacceecd}$ & $\mathtt{\{1,a,c\}}$   \\[-1.8ex]
     &   & $\mathtt{cdefefdedeffffcceebbddffef}$ & $\mathtt{\{2,e,f\}}$    \\  \hline
     &   & $\mathtt{000111222333444}$ & $\mathtt{\{0,3,b,d\}}$  \\[-1.8ex]
  $5$& 1 & $\mathtt{abcabcabdaceabd}$ & $\mathtt{\{1,4,a,c\}}$ \\[-1.8ex]
     &   & $\mathtt{dfeedfcefbdffce}$ & $\mathtt{\{2,e,f\}}$  \\ \hline
\end{tabular}
\end{center}
\end{proof}

The \emph{leave} of a partial BIBD $(V,\mathcal{B})$ with index $1$ is the graph with vertex set
$V$ in which a pair of vertices is adjacent if and only if the pair is not contained in any block
in $\mathcal{B}$. The next lemma is very similar to a result in \cite{FuLiRo1} and is used only in
the proof of Lemma \ref{AtLeastCChromaticMPT}.

\begin{lemma}\label{NotTripoleDecomp}
Let $v$ and $m$ be positive integers such that $v \equiv 4 \mod 6)$, $m \equiv 0 \mod 6)$, and $24
\leq m < v$. There exists a partial $(v,3,1)$-BIBD whose leave has a decomposition into $m$
perfect matchings on $v$ vertices, and the vertex-disjoint union of a complete graph of order $4$
and a perfect matching on $v-4$ vertices.
\end{lemma}

\begin{proof}[{\bf Proof}]
The statement of Lemma 6.5 of \cite{FuLiRo1} gives a decomposition of a complete graph of order
$v$ into triangles, $m$ perfect matchings on $v$ vertices, and the vertex-disjoint union of a copy
of $K_{1,3}$ and a perfect matching on $v-4$ vertices. Furthermore, in each case the construction
given in the proof contains a triangle whose vertices are the vertices of degree $1$ in the copy
of $K_{1,3}$ (this arises through the use of Lemma 6.3 of \cite{FuLiRo1}, and the copy of
$K_{1,3}$ and the triangle are given explicitly in the proof of that result). Taking this
decomposition and removing the perfect matchings, the copy of $K_{1,3}$ and the special triangle
gives the required BIBD.
\end{proof}

\begin{lemma}\label{AtLeastCChromaticMPT}
Let $c$ be a positive integer such that $c \geq 3$. For all sufficiently large even integers $v$
there exists a partial $(v,3,1)$-BIBD which has chromatic number at least $c$ and whose leave is
\begin{itemize}
    \item
a perfect matching on $v$ vertices if $v \equiv 0,2 \mod 6)$; and
    \item
the vertex-disjoint union of a complete graph of order $4$ and a perfect matching on $v-4$
vertices if $v \equiv 4 \mod 6)$.
\end{itemize}
\end{lemma}

\begin{proof}[{\bf Proof}] It follows from Lemma \ref{CChromaticPartialBIBD} that for some positive integer $u$ there
is a partial $(u,3,1)$-BIBD $(U,\mathcal{A})$ which has chromatic number $c$. We can assume that $u
\geq 12$ by adding points to this BIBD, if necessary. By the main result of \cite{BrHo}, we can
embed $(U,\mathcal{A})$ in a $(u',3,1)$-BIBD $(U',\mathcal{A}')$ for some positive integer $u'$
such that $2u+1 \leq u' \leq 2u+5$ and $u' \equiv 1 \mod 6)$.

Let $v$ be an even integer such that $v \geq 2u'+2$. If $v \equiv 0,2 \mod 6)$, the main result of
\cite{FuLiRo2} then guarantees that this design can in turn be embedded in a partial
$(v,3,1)$-BIBD whose leave is a perfect matching on $v$, so we may assume that $v \equiv 4 \mod
6)$. Let $U'=\{x_0,x_1,\ldots,x_{u'-1}\}$ and let $W$ be a set of size $v-u'$ which is disjoint
from $U'$. By Lemma \ref{NotTripoleDecomp} there is a partial $(v-u'+1,3,1)$-BIBD $(W \cup
\{x_0\},\mathcal{A}'')$ whose leave has a decomposition into $u'-1$ perfect matchings
$F_1,F_2,\ldots,F_{u'-1}$ on $v-u'+1$ vertices, and the vertex-disjoint union of a complete graph
of order $4$ and a perfect matching on $v-u'-3$ vertices. Then $(U \cup W,\mathcal{B})$, where
    $$\mathcal{B}=\mathcal{A}' \cup \mathcal{A}'' \cup \{(x_i,y,z): yz \in E(F_i), i \in
    \{1,2,\ldots,u'-1\}\},$$
can be seen to be a $(v,3,1)$-BIBD which has chromatic number at least $c$ and whose leave is the
vertex-disjoint union of a complete graph of order $4$ and a perfect matching on $v-4$ vertices.
\end{proof}

\begin{proof}[{\bf Proof of Theorem \ref{MainTheorem} in the case $\boldsymbol{k=3}$.}] Let $N$ be the smallest even
integer such that for each even integer $u' \geq N$ there exists a partial $(u',3,1)$-BIBD which
has chromatic number at least $c$ and whose leave satisfies the conditions of Lemma
\ref{AtLeastCChromaticMPT}.

Let $v$ be a $(3,\lambda)$-admissible integer such that $v \geq 3N$. We will show that there exists
a $c$-chromatic $(v,3,\lambda)$-BIBD. Let $u$ and $h$ be the integers such that $v=3u+h$, $u$ is
even, and $0 \leq h \leq 5$. Then $u \geq N$ and there is a partial $(u,3,1)$-BIBD
$(U,\mathcal{A})$ which has chromatic number at least $c$ and whose leave satisfies the conditions
of Lemma \ref{AtLeastCChromaticMPT}. Let $\mathcal{A} = \{A_1,A_2,\ldots,A_t\}$. If $u \equiv 0,2
\mod 6)$, then let $P^*$ be a pair of points in $U$ which are adjacent in the leave of
$(U,\mathcal{A})$. If $u \equiv 4 \mod 6)$, then let $P^*$ be the set of the four points in $U$
which are mutually adjacent in the leave of $(U,\mathcal{A})$. In either case, let $\mathcal{P}$ be
a partition of $U \setminus P^*$ into pairs of points such that each pair is adjacent in the leave
of $(U,\mathcal{A})$.

Let $Z=\{z_1,z_2,z_3\}$ be a set and let $H$ be a set such that $|H|=h$. Let $\{H_1,H_2,H_3\}$ be a
partition of $H$ such that any two of $|H_1|$, $|H_2|$ and $|H_3|$ differ by at most $1$. Let $V=(U
\times Z) \cup H$ be a point set. Let $S_i=(U \times \{z_i\}) \cup H_i$ for each $i \in \{1,2,3\}$.
We will construct collections of blocks $\mathcal{C}_0,\mathcal{C}_1,\ldots,\mathcal{C}_t$ such
that
\begin{itemize}
    \item[(i)]
$(V,\mathcal{C}_i)$ is a $(v,3,\lambda)$-BIBD for each $i \in \{0,1,\ldots,t\}$;
    \item[(ii)]
$\{S_1,S_2,S_3\}$ is a blocking system for $(V,\mathcal{C}_0)$;
    \item[(iii)]
$(V,\mathcal{C}_t)$ contains an isomorphic copy of $(U,\mathcal{A})$;
    \item[(iv)]
the chromatic number of $(V,\mathcal{C}_{i+1})$ is at most one more than the chromatic number
of $(V,\mathcal{C}_i)$ for each $i \in \{0,1,\ldots,t-1\}$.
\end{itemize}
From (ii) it will follow that $(V,\mathcal{C}_0)$ has chromatic number at most 3, and from (iii) it
will follow that $(V,\mathcal{C}_t)$ has chromatic number at least $c$. Thus, from (iv) it will
follow that $(V,\mathcal{C}_j)$ has chromatic number $c$ for some $j \in \{0,1,\ldots,t\}$. So it
suffices to find such collections of blocks.

For each $A \in \mathcal{A}$, let $\mathcal{B}^{\dag}_A$ and $\mathcal{B}_A$ be collections of
blocks such that $(A \times Z,\{\{x\} \times Z:x \in A\},\mathcal{B}^{\dag}_A)$ and $(A \times
Z,\{\{x\} \times Z:x \in A\},\mathcal{B}_A)$ are $(3,\lambda)$-GDDs of type $3^3$ such that
\begin{itemize}
    \item
$(A \times \{z_1\},A \times \{z_2\},A \times \{z_3\})$ is a blocking system for $(A \times
Z,\mathcal{B}^{\dag}_A)$;
    \item
$A \times \{z_1\} \in \mathcal{B}_A$; and
    \item
there are two distinct points $x,y \in A\times Z$ such that every block which is in
$\mathcal{B}_A$ but not in $\mathcal{B}^{\dag}_A$ contains either $x$ or $y$;
\end{itemize}
(such collections exists by Lemma \ref{TwoTDsBlockSize3}, taking $\lambda$ copies of every block of
the transversal designs). For each $P \in \mathcal{P}$, let $\mathcal{B}_P$ be a collection of
blocks such that $((P\times Z) \cup H,\{H\} \cup \{\{x\}:x \in P \times Z\},\mathcal{B}_P)$ is a
$(3,\lambda)$-GDD of type $h^11^6$ for which $\{(P \times \{z_1\}) \cup H_1,(P \times \{z_2\}) \cup
H_2,(P \times \{z_3\}) \cup H_3\}$ is a blocking system (such a collection exists by Lemma
\ref{HoleyBIBDBlockSize3}). Let $\mathcal{B}_{P^*}$ be a collection of blocks such that
$((P^*\times Z) \cup H,\mathcal{B}_{P^*})$ is a $(3|P^*|+h,3,\lambda)$-BIBD for which $\{(P \times
\{z_1\}) \cup H_1,(P \times \{z_2\}) \cup H_2,(P \times \{z_3\}) \cup H_3\}$ is a blocking system
(such a collection exists by Lemma \ref{BIBDBlockSize3} since $v$ is $(3,\lambda)$-admissible and
$3|P^*|+h \equiv v \mod 6)$).

For each $k \in \{0,1,\ldots,t\}$, let
$$\mathcal{C}_k = \left(\bigcup_{i=1}^{k} \mathcal{B}_{A_i} \right) \cup \left(\bigcup_{i=k+1}^{t} \mathcal{B}^{\dag}_{A_i} \right) \cup \left(\bigcup_{P \in \mathcal{P}} \mathcal{B}_P \right) \cup \mathcal{B}_{P^*}.$$
It only remains to show that (i), (ii), (iii) and (iv) hold.

It is routine to check that (i), (ii) and (iii) hold (for (iii), the isomorphic copy of
$(U,\mathcal{A})$ is on the point set $U \times \{z_1\}$). To see that (iv) holds, let $i \in
\{0,1,\ldots,t-1\}$ and let $c_i$ be the chromatic number of $(V,\mathcal{C}_i)$. There are two
points $x$ and $y$ of $V$ such that every block which is in $\mathcal{B}_{A_{i+1}}$ but not in
$\mathcal{B}^{\dag}_{A_{i+1}}$ contains either $x$ or $y$. This implies that every block which is
in $\mathcal{C}_{i+1}$ but not in $\mathcal{C}_i$ contains either $x$ or $y$. Thus, we can obtain
a $(c_i+1)$-colouring of $(V,\mathcal{C}_{i+1})$ by taking a $c_i$-colouring of
$(V,\mathcal{C}_i)$ and recolouring the vertices $x$ and $y$ with a colour which is not used in
the original colouring.
\end{proof}

\vspace{0.3cm} \noindent{\bf Acknowledgements}

The first author was supported by an AARMS postdoctoral fellowship and by Australian Research
Council grants DE120100040 and DP120103067. The second author was supported by research grants from
NSERC, CFI and IRIF. The authors would like to thank the referees, whose comments substantially
improved this paper.

\end{document}